\documentclass[a4paper,11pt]{amsart}

\usepackage[dvips]{graphics,color}
\usepackage{multicol}
\usepackage{amsmath,latexsym,amsbsy,amssymb}
\usepackage{psfrag,graphicx}
\usepackage{enumerate}
\usepackage{amsthm}
\usepackage[latin1]{inputenc}

\newtheorem{Theorem}{Theorem}[section]
\newtheorem{Definition}[Theorem]{Definition} 
\newtheorem{Proposition}[Theorem]{Proposition}
\newtheorem{Corollary}[Theorem]{Corollary}

\newtheorem{Lemma}[Theorem]{Lemma} 
\newtheorem{Remark}[Theorem]{Remark}
\newtheorem{Example}[Theorem]{Example}

\numberwithin{equation}{section}

\newcommand{\R}{{\mathbb R}}
\newcommand{\N}{{\mathbb N}}
\begin{document}

\title[Regularity of the minimum time function]{Variational analysis and regularity of  the minimum time function for differential inclusions}

\author[Luong V. Nguyen]{Luong V. Nguyen}
\address[Luong V. Nguyen]{Institute of Mathematics, Polish Academy of Sciences, ul. \'Sniadeckich 8, 00 - 656  Warsaw, Poland \\Telephone: +48 22 5228 235 }

\email{vnguyen@impan.pl; luongdu@gmail.com}



\date{\today}
\begin{abstract}
We study the time optimal control problem for differential inclusions with a general closed target. We first give a representation of the proximal horizontal subgradients of the minimum time function $\mathcal{T}$ and then, together with a known representation of the proximal subgradients, we obtain some relationships between the normal cones to the sublevel set and to the epigraph of  $\mathcal{T}$. These relationships allow us to get the propagation of the proximal subdifferential as well as of the proximal horizontal subdifferential of $\mathcal{T}$ along optimal trajectories. Finally, we show, under suitable assumptions, that the epigraph of $\mathcal{T}$ is $\varphi$-convex near the target. This is the first nonlinear $\varphi$-convexity result valid in any dimension.
\end{abstract}
\maketitle

\section{Intoduction}
Let $F:\R^n \rightrightarrows \R^n$ be a Lipschitz continuous sublinear multifunction and $\mathcal{K}$ be a closed subset of $\R^n$.  We consider the minimum time function associated to the target $\mathcal{K}$ for the differential inclusion
 \begin{equation}\label{IDI}
 \left\{
  \begin{array}{lcl}
    \dot{x}(t) & \in &F(x(t)),\quad\quad \mathrm{a.e.}\, t>0 \\
    x(0)& = & x_0 \in \R^n
  \end{array}
  \right.
 \end{equation}
A trajectory of $F$ starting from $x_0$ is an absolutely continuous function $x(\cdot)$ defined on $[0,+\infty)$ satisfying (\ref{IDI}), i.e., $\dot{x}(t) \in F(x(t))$ for a.e. $t>0$ and $x(0) = x_0$. Here, the notion $\dot{x}(t)$ refers to the derivative of $x(\cdot)$ at the time $t$ and it is the right derivative if $t = 0$.

The \textit{minimum time function} for the differential inclusion (\ref{IDI}) associated to the target $\mathcal{K}$ is defined as follows:  for $x_0\in \R^n$,
$$\mathcal{T}(x_0) := \inf\{t>0: \exists\,x(\cdot)\,\text{satisfying}\,(\ref{IDI})\,\text{with}\,x(0) = x_0\,\,\text{and}\,\,x(t)\in \mathcal{K}\},$$
with the convention $\inf \emptyset = +\infty$. When $\mathcal{T}(x_0)$ is finite, it is  the minimal time taken by the trajectories of $F$ starting from $x_0$ to reach the target $\mathcal{K}$. The set $\mathcal{R}$ of points $x\in \R^n$ such that $\mathcal{T}(x) <+\infty$ is called the \textit{controllable set}.

The regularity of the minimum time function $\mathcal{T}$ is a classical and widely studied topic in control theory. It is related to the controllability properties of the control systems as well as to the regularity of the target and the dynamics, together with suitable relations between them. It is well-known that $\mathcal{T}$ is locally Lipschitz in $\mathcal{R}$ if and only if \textit{Petrov's controllability condition} is satisfied (see, e.g., \cite{CS0,CS,VO,WYu}). However, in general, $\mathcal{T}$ is not everywhere differentiable even for smooth data. The strongest regularity property for $\mathcal{T}$ that we can expect, in fairly general cases,  is \textit{semiconcavity}. Here, a function is said to be semiconcave if it can be written as a sum of a $C^2$ function and a concave function. Hence, semiconcave functions inherit many fine properties from concave functions. In this case, $\mathcal{T}$ is locally Lipschitz and a.e. twice differentiable. Cannarsa and Sinestrari showed in \cite{CS0} that the minimum time function is locally semiconcave in $\mathcal{R}\setminus \mathcal{K}$ if Petrov's condition holds and the target satisfies a \textit{uniform interior sphere condition}, i.e., there exists $r>0$ such that for any $x\in \mathcal{K}$, there exists $y\in \mathcal{K}$ such that $x\in \bar{B}(y,r) \subset \mathcal{K}$. Due to the equivalence between Petrov's condition and the Lipschitz continuity,  $\mathcal{T}$ is no longer semiconcave if we remove Petrov's condition. Therefore, it is natural to study the structure of the minimum time function under controllability assumptions which are weaker than Petrov's condition. Keeping uniform interior sphere conditon of $\mathcal{K}$ and  assuming the continuity of $\mathcal{T}$ and the \textit{pointedness} of the normal cones to the hypograph, Colombo and Nguyen showed in \cite{CK} that the \textit{hypograph} of $\mathcal{T}$ is $\varphi$-convex for a suitable continuous function $\varphi$.  This kind of regularity is weaker than semiconcavity. However, $\mathcal{T}$ keeps many regularity properies of semiconcave functions. The proof of the $\varphi$-convexity for the hypograph of $\mathcal{T}$  in \cite{CK} is based on representations of the proximal superdifferential and proximal horizontal superdifferential of $\mathcal{T}$. Removing the pointedness assumption, Nguyen showed in \cite{K} that $\mathcal{T}$  still enjoys good regularity although the hypograph of $\mathcal{T}$ satisfies a weaker regularity called \textit{exterior sphere condition}.

It is worth remarking that all regularity results obtained in \cite{CS0,CK,K} are dealt with the case where $F$ is given in the form of a $C^{1,+}$ parameterization
$$F(x) = \{f(x,u): u\in U\},\qquad x\in \R^n,$$
with $U\subset \R^m$ compact and $f:\R^n\times U \to \R^n$ is of class $C^{1,+}$. However, it is difficult to know when multifunctions admit smooth parameterizations (see \cite{PCPW} for a discussion). To get rid of finding smooth parameterizations for $F$, in \cite{PCPW}, Cannarsa and Wolenski developed a new approach, based on the nonsmooth maximum principle, to obtain semiconcavity results of the value function of a \textit{Mayer problem} for the differential inclusion (\ref{IDI}). One essential assumption for this approach is the semiconvexity in the first variable of the \textit{maximized Hamiltonian} $H$ associated to $F$:
$$H(x,p): = \sup_{v\in F(x)}\langle v,p\rangle,\qquad\qquad (x,p) \in \R^n \times \R^n.$$
Adapting this approach to the optimal time problem, Cannarsa, Marino and Wolenski obtained semiconcavity results of $\mathcal{T}$ for (\ref{IDI}) keeping the interior sphere condition of $\mathcal{K}$ and Petrov's condition (see\cite{CMW13}). Later, some results for smooth parameterized control system were extended to nonparameterized  systems (see, e.g., \cite{camapw,CMW13,CaKh,PCAMKN,CASC}). In particular, in \cite{CaKh}, Cannarsa and Nguyen extended the analysis of \cite{CK,K} to the general system (\ref{IDI}). More precisely, assuming the continuity of $\mathcal{T}$, they showed that the hypograph of $\mathcal{T}$ satisfies an exterior sphere condition provided either $\mathcal{K}$ or $F(x)$ satisfies an interior sphere condition for all $x\in \R^n$.

In contrast to the semiconcavity type, there are few papers dealing with the semiconvexity type of the minimum time function. It was shown in \cite{CS0} that the minimum time function for linear systems is semicovex if the target is convex and Petrov's condition holds. Again for linear systems and convex targets,  removing Petrov's condition but assuming the continuity of $\mathcal{T}$, Colombo, Marigonda and Wolenski showed in \cite{CMW} that the \textit{epigraph} of $\mathcal{T}$ is $\varphi$-convex. Then $\mathcal{T}$ satisfies many good properties as listed in Proposition \ref{ProPRF}. Furthermore, in \cite{CK13}, the authors proved, for two dimensional nonlinear affine control systems and $\mathcal{K}=\{0\}$, that the epigraph of $\mathcal{T}$  is $\varphi$-convex in a small neighborhood of the origin. The proof relies heavily on the (strictly) convexity of sublevel sets of $\mathcal{T}$ (in small time) and on the fact that every point sufficiently close to the origin is \textit{optimal}, i.e., any trajectory steering  a point to the origin optimally can be extended backward still remaining optimal.   To the best of my knowledge, there is no such type of regularity results in a more general setting where $n>2$ for a nonlinear system or where the target is not a single point. In this paper, we will show,  under suitable assumptions, that the epigraph of $\mathcal{T}$, for nonparameterized control system (\ref{IDI}), is $\varphi$-convex near the target (see Theorem 5.7).  More precisely, we prove that if sublevel sets of $\mathcal{T}$ are uniformly $\varphi_0$-convex for some constant $\varphi_0\ge 0$, then there exists a suitable continuous function $\varphi$ such that the epigraph of $\mathcal{T}$ is $\varphi$-convex. Note that, in the proof, we do not need the optimality of points near the target. Futhermore, the proof is also based on some sensitivity relations. 

Sensitivity relations are an interesting and important object in control theory because of applications to optimality conditions, optimal synthesis and regularity of the value function. These relations consist of the \textit{dual arc} satisfying an inclusion of an appropriate generalized gradient of the value function. For the minimal time problem, the first results were presented in \cite{CFS} which dealt with the smooth parameterized systems and the target having an interior sphere condition. In fact, for an optimal trajectory $x(\cdot)$ starting at a point $x_0 \in \mathcal{R}$, they proved that there exists (by maximum principle) a dual arc $p(\cdot)$ such that $p(t)$ belongs to the Fr\'echet superdifferential of $\mathcal{T}$ at $x(t)$ for all $t\in [0,\mathcal{T}(x_0))$ if Petrov's condition holds true at the end point $x(\mathcal{T}(x_0))$. This result was extended to nonparameterized systems in \cite{CMW13} by a different approach. It was proved in \cite{PCAMKN} for nonparameterized systems  that if Petrov's condition holds at the end point $x(\mathcal{T}(x_0))$ then, for all $t\in [0,\mathcal{T}(x_0))$, $p(t)$ belongs to the proximal superdifferential of $\mathcal{T}$ at $x(t)$, otherwise $p(t)$ belongs to the proximal horizontal superdifferential of $\mathcal{T}$ at $x(t)$ for all $t\in [0,\mathcal{T}(x_0))$. Recently, in \cite{CASC}, Cannarsa and Scarinci recovered the results of \cite{PCAMKN} for a general target. They also proved analogous inclusion for the proximal subdifferential  extending the result, for smooth parameterized systems, obtained in \cite{HL14}. More precisely, they showed that the proximal subdifferential of $\mathcal{T}$ propagates along optimal trajectories except at the terminal points. In the present paper, we obtain similar propagation results for both  proximal subdifferential and  proximal  horizontal subdifferential of $\mathcal{T}$ (Corollary \ref{Co2} and \ref{Co3}): we show that  proximal subdifferential and  proximal  horizontal subdifferential of $\mathcal{T}$ propagate wholly along  optimal trajectories. These are consequences of Theorem \ref{THM6} and \ref{THM7} where we prove inclusions for normal cones to the epigraph  and to the sublevel sets of the minimum time function. The proofs of these results are based on a relationship between normals to the epigraph and to sublevel sets of $\mathcal{T}$ via the value at relevant points of the \textit{minimized Hamiltonian} $h$ associated to $F$:
$$h(x,p) := \inf_{v\in F(x)}\langle v,p\rangle, \qquad\qquad (x,p) \in \R^n \times \R^n.$$

 It is proved in \cite{CKL}, for nonlinear control systems, that if  $x\in \mathcal{R}$ and if $\zeta $ belongs to the normal cone of the sublevel set $\mathcal{R}(\mathcal{T}(x)) := \{y\in \R^n: \mathcal{T}(y)\le \mathcal{T}(x)\}$ at $x$ then $(\zeta,h(x,\zeta))$ is a normal to the epigraph of $\mathcal{T}$ at $(x,\mathcal{T}(x))$. The proof is based on Maximum Principle.  Note that, in that paper, besides standard assumptions, it is assumed, in a neighborhood of $x$, that $\mathcal{T}$ is continuous, optimal controls are unique and \textit{bang - bang} with finitely many \textit{switchings}, the sublevel sets are $\varphi$-convex and every point is an \textit{optimal point}.  Under the same assumptions, the reversed implication is also proved in \cite{LTS}, namely if $(\zeta,\alpha)$ is a normal to the epigraph of $\mathcal{T}$ at $(x,\mathcal{T}(x))$ then $\zeta$ is a normal to $\mathcal{R}(\mathcal{T}(x))$ and $h(x,\zeta)= \alpha$. In the present paper, we prove the same conclusions for very gerenal differential inclusions without using maximum principle. The proof is based on the representations of proximal horizontal subdifferential (Theorem \ref{THM1}) and proximal subdifferential of $\mathcal{T}$ (Theorem 5.1 in \cite{WYu}). Moreover,  in Section \ref{SectV} we prove a special feature of the minimum time function, that is, the normal cones to the epigraph of $\mathcal{T}$ at $(x,\mathcal{T}(x))$ and to the sublevel $\mathcal{R}(\mathcal{T}(x))$ at $x\in \mathcal{R}$ have the same dimension.
 
 The paper is organized as follows. In Section \ref{S2} we recall some notions and  preliminary results needed in the sequel. Section \ref{SectV} is devoted to the variational analysis for the minimum time function. Section \ref{SectS} concerns with sensitivity relations. The regularity of the minimum time function is studied in Section \ref{SectR}.
 \section{Preliminaries}\label{S2}
 \subsection{Notations and basic facts}
In this section we recall some basic concepts of nonsmooth analysis. Standard references are in \cite{CLSW,RO}.\\
We denote by $|\cdot|$ the Euclidean norm in $\R^n$,  by $\langle\cdot,\cdot\rangle$ the inner product and by $[x,y]$ the segment connecting two points $x$ and $y$ in $\R^n$. We also denote by $B(x,r)$ the open ball of radius $r>0$ centered at $x$, $\mathbb{S}^{n-1}$ the unit sphere in $\R^n$, and by $\mathbb{M}_{n\times m}(\R)$ the set of all matrices of $n$ rows and $m$ columns. We will use the shortened $\mathbb{B} = B(0,1)$. For any subset $E$ of $\R^n$, we denote by $\mathrm{bdry}E$ its boundary, by $\bar{E}$ its closure, by $\mathrm{co}E$ its convex hull and by $\bar{co}E$ its closed convex hull. A subset $C$ of $\R^n$ is called a cone if and only if $\lambda x \in C$ for any $x\in C$ and $\lambda \ge 0$. We say that $\kappa \in \N$ is the dimension of a cone $C$ if there exist $v_1,\cdots,v_\kappa \in C$ such that they are linearly independent and for any $v\in C$ there exist $\lambda_1,\cdots,\lambda_\kappa \ge 0$ such that $v = \lambda_1 v_1 +\cdots + \lambda_\kappa v_\kappa$.
\\Let $K\subset\R^n$ be a closed subset with boundary $\mathrm{bdry}K$. Denote by $\mathrm{proj}_K(x)$ the projection of $x\in \R^n$ on $K$. Given $x\in K$ and $v\in \R^n$. We say that $v$ is a \textit{proximal normal} to $K$ at $x$ if there exists $\sigma := \sigma(x,v) \ge 0$ such that
\begin{equation}
\label{DefProx}
\langle v,y-x\rangle \le \sigma|y-x|^2, \quad\text{for all}\, y\in K.
\end{equation} 
 We denote the set of all proximal normals to $K$ at $x$ by $N^P_K(x)$ and call it the \textit{proximal normal cone} to $K$ at $x$.
 
 Equivalently, $v\in N^P_K(x)$ if there exist constants $C >0$ and $\eta >0$ such that  
\begin{equation*}
\langle v,y-x\rangle \le C|y-x|^2, \quad\text{for all}\, y\in B(x,\eta) \cap K .
\end{equation*}  
 
Observe that $v\in N^P_K(x)$ if and only if there is some $\lambda>0$ such that $\mathrm{proj}_K(x+\lambda v) = \{x\}$.  Notice that if $K$ is convex, we can take $\sigma = 0$ in (\ref{DefProx}). Hence the proximal normal cone to $K$ at $x$ reduces to the normal cone in the sense of Convex Analysis.

The Clarke normal cone to $K$ at $x$, $N^C_K(x)$, is defined as
$$N^C_K(x) = \bar{co}\{ v\in \R^n: \exists x_i \to x, \exists v_i \to v, v_i \in N^P_K(v_i)  \}.$$
Let $\Omega$ be an open set of $\R^n$ and let $f: \Omega \to \R \cup \{+\infty\}$ be a lower semicomtinuous function. The \textit{domain} of $f$ is the set $\mathrm{dom}(f) := \{x\in \Omega: f(x) <+\infty\}$, the\textit{ epigraph} of $f$ is the set
 $\mathrm{epi}(f) := \{(x,\beta) \in \Omega \times \R: x\in \mathrm{dom}(f), \beta \ge f(x)\}$. Let $x\in \mathrm{dom}(f)$.
 \begin{itemize}
 \item The \textit{proximal subdifferential} of $f$ at $x$ is the set
 \begin{equation*}
 \partial^P f(x) := \{ v \in\ R^n: (v,-1) \in N^P_{\mathrm{epi}(f)}(x,f(x))\}.
 \end{equation*}
 Equivalently,
 \begin{equation*}
 \partial^Pf(x) = \{v\in\R^n: \exists c,\rho >0 \,\,\mathrm{s.t} \,\,f(y) - f(x)-\langle v,y-x\rangle \ge -c|y-x|^2,\,\forall y\in B(x,\rho)\}.
 \end{equation*}
 An element of $\partial^Pf(x)$ is called a proximal subgradient of $f$ at $x$.
 \item The \textit{horizontal proximal subdifferential} of $f$ at $x$ is the set
 \begin{equation*}
 \partial^{\infty}f(x) := \{v\in\R^n: (v,0) \in N^P_{\mathrm{epi}(f)}(x,f(x))\}.
 \end{equation*}
 An element of $\partial^{\infty}f(x)$ is called a proximal horizontal subgradient of $f$ at $x$.
 \item The \textit{Fr\'echet subdifferential} of $f$ at $x$ is the set 
 $$\partial^{-} f(x) := \left\{v\in\R^n: \liminf_{y\to x} \frac{f(y) - f(x) - \langle v,y-x\rangle}{|y-x|} \ge 0    \right\}.$$
 An element of $\partial^{-}f(x)$ is called a Fr\'echet subgradient of $f$ at $x$.
 \item The \textit{Fr\'echet superdifferential} of $f$ at $x$ is the set 
 $$\partial^{+} f(x) := \left\{v\in\R^n: \limsup_{y\to x} \frac{f(y) - f(x) - \langle v,y-x\rangle}{|y-x|} \le 0    \right\}.$$
  An element of $\partial^{+}f(x)$ is called a Fr\'echet supergradient of $f$ at $x$.
 \end{itemize}
  Assume that $f$ is Lipschitz around $x$. The \textit{Clarke's generalized gradient} of $f$ at $x$ is defined by
 $$\partial f(x) := \mathrm{co}\left\{v\in \R^n: \exists \{y_i\} \subset \Omega\,\,\mathrm{s.t.}\,\,f\,\,\text{is differentiable at}\,\,y_i, y_i\to x,\nabla f(y_i)\to v    \right\} $$
For a mapping $G:\R^n \times \R^m \to \R$ associating to $x\in\R^n$ and $y\in \R^m$ a real number, we will denote by $\nabla_xG$, $\nabla_yG$ the partial gradients (when they exist), and by $\partial_xG$, $\partial_yG$ the partial generalized gradients.

 Let $\Omega \subset \R^n$ be open. A function $f:\Omega \to \R$ is called \textit{semiconcave} with semiconcavity constant $c \ge 0$ if $f$ is continuous on $\Omega$ and satisfies
 $$f(x+h) + f(x-h) -2f(x) \le c|h|^2$$
 for all $x,h\in\R^n$ such that $[x-h,x+h] \subset \Omega$. We say that a function $g:\Omega \to \R$ is \textit{semiconvex} if and only if $-g$ is semiconcave. We recall below some useful properties of semiconcave functions.
 \begin{Proposition}\label{ProSC}
 Let $\Omega\subset \R^n$ be open, $f:\Omega \to \R$ be a semiconcave function with semiconcavity constant $c$ and let $x\in \Omega$. Then $f$ is locally Lipschitz on $\Omega$ and the followings hold true
 \begin{itemize}
 \item[(1)] $p\in \R^n$ belongs to $\partial^{+}f(x)$ if and only if for any $y\in \Omega$ such that $[y,x]\subset \Omega$,
 \begin{equation} \label{Eq21}
 f(y) - f(x) -\langle p,y-x\rangle \le  c|y-x|^2
 \end{equation}
 \item[2)] $\partial f(x) = \partial^{+} f(x)$.
\end{itemize}  
 \end{Proposition}
 If $f$ is semiconvex, then (\ref{Eq21}) holds with the reversed inequality and the reversed sign of the quadratic term and the statement (2) holds true with the subdifferential instead of the superdifferential. For further properties and characterizations of semiconcave/semiconvex functions, we refer the reader to \cite{CS}.
 \begin{Definition}
 Suppose $K\subset \R^n$ is closed and $\varphi: K\to [0,+\infty)$ is continuous. We say that $K$ is $\varphi$-convex if
 \begin{equation} \label{Eq22}
 \langle v,y-x\rangle \le \varphi(x)|v||y-x|^2,
 \end{equation}
 for all $x,y\in K$ and $v\in N^P_K(x)$.
 \end{Definition}
 The case when $\varphi \equiv 0$ in (\ref{Eq22}) is equivalent to the convexity of $K$. Therefore, $\varphi$-convexity is a generalization of convexity. Moreover, if the boundary of $K$ is the graph of a $C^{1,1}$ function then $K$ is $\varphi$-convex with $\varphi$ is a suitable constant function. Functions whose epigraph is $\varphi$-convex enjoy good regularity properties which are similar to properties of convex functions. Denote by $\mathcal{L}^n$ and $\mathcal{H}^d$ the Lebesgue $n$-dimensional measure and the Hausdorff $d$-dimensional measure, respectively. We recall here some regularity properties of functions whose epigraph is $\varphi$-convex (see \cite{GCAM}).
 \begin{Proposition}  \label{ProPRF}
 Let $\Omega \subset \R^n$ be open and let $f:\Omega \to \R$ be continuous and such that $\mathrm{epi}(f)$ is $\varphi$-convex for some suitable continuous function $\varphi$. Then there exists a sequence of sets $\Omega_h\subset \Omega$ such that $\Omega_h$ is compact in $\Omega$ and
 \begin{itemize}
 \item[(i)] the union of $\Omega_h$ covers $\mathcal{L}^n$-almost all $\Omega$;
 \item[(ii)] for all $x\in \cup_{h}\Omega_h$, there exist $\delta=\delta(x)>0$, $L=L(x)>0$ such that $f$ is Lipschitz on $B(x,\delta)$ with ratio $L$, and hence semiconvex on $B(x,\delta)$.
 \end{itemize}
 Consequently,
 \begin{itemize}
 \item[(iii)] $f$ is a.e. Fr\'echet differentiable and admits a secon order Taylor expansion around a.e. point of its domain.
 \end{itemize}
 Moreover, the set of points where the graph of $f$ is nonsmooth has small Hausdorff dimension. More precisely,
 \begin{itemize}
 \item[(iv)] for every $k=1,\cdots,n$, the set $\{x\in \mathrm{int}\mathrm{dom}(f): \dim\partial^Pf(x)\ge k\}$ is countably $\mathcal{H}^{n-k}$-rectifiable.
 \end{itemize}
 Finally,
 \begin{itemize}
 \item[(v)] $f$ is of locally bounded variation in $\Omega$.
 \end{itemize}
 \end{Proposition}
 \subsection{Differential inclusions and the minimum time function}
 Let $F:\R^n \rightrightarrows \R^n$ be a given multifunction. We consider the differential inclusion, for $T>0$,
  \begin{equation}\label{DI}
 \left\{
  \begin{array}{lcl}
    \dot{x}(s) & \in &F(x(t)),\quad\quad \mathrm{a.e.}\,\, t\in [0,T] \\
    x(0)& = & x_0 \in \R^n
  \end{array}
  \right.
 \end{equation}
 A \textit{solution}  of (\ref{DI}) is an absolutely continuous function $x(\cdot)$ defined on $[0,T]$ with initial value $x(0) = x_0$. We also say that $x(\cdot)$ is a \textit{trajectory} of $F$ starting at $x$. The notion $\dot{x}(t)$ refers to the derivative of $x(\cdot)$ at the time $t$ and it is the right derivative  if $t=0$.
 \\Throughout this paper, we require the following assumptions on the multifunction $F$.\\
 \textbf{Assumption (F).}
 \begin{itemize}
 \item[(F1)] $F(x)$ is nonempty, convex, and compact for each $x\in \R^n$.
 \item[(F2)] $F$ is locally Lipschitz, i.e. for each compact set $K \subset \R^n$, there exists a constant $L>0$ such that
 $$F(x) \subset F(y) +L|y-x|\bar{\mathbb{B}},\quad\text{for all}\, x,y\in K.$$
 \item[(F(3)] there exists $\gamma >0$ such that $\mathrm{max}\{|v|: v\in F(x)\} \le \gamma(1+|x|)$, for all $x\in \R^n$.
 \end{itemize}
 The following theorem gives some information regarding $C^1$ trajectories of $F$ under assumption (F) which will be useful in the sequel
 \begin{Theorem}[see, e.g., \cite{WYu}] \label{C1}
 Assume that assumption (F) holds true. Let $K\subset \R^N$ be compact. Then there exists $T>0$ such that associated to every $x\in K$ and $v\in F(x)$ is a trajectory $x(\cdot)$ defined on $[0,T]$ with $\dot{x}(0) = v$. Moreover, for all $t\in [0,T]$, we have $|\dot{x}(t) - v| \le Mt$, for some constant $M$ independent of $x$.
 \end{Theorem}
 We now assume that a closed subset $\mathcal{K}$ of $\R^n$ is given which is called the target and $F:\R^n \rightrightarrows \R^n$ is a multifunction. We define the minimum time function $\mathcal{T}: \R^n \to [0,+\infty]$ as follows. If $x\not\in\mathcal{K}$ then
 \begin{equation}
 \label{DefT}
 \mathcal{T}(x): = \inf\{T>0:\exists x(\cdot)\,\,\text{satisfying}\,\,(\ref{DI})\,\,\text{with}\,\,x(0) =x\,\,\text{and}\,\,x(T) \in \mathcal{K}\}.
 \end{equation}
 If there is no trajectory of $F$ starting at $x$ can reach $\mathcal{K}$, then  $\mathcal{T}(x) = +\infty$ as the usual convention. If $x\in \mathcal{K}$ then we set $\mathcal{T}(x)=0$.
 \\It is well-known that, under assumption (F), the infimum in (\ref{DefT}) is attained and the minimum time function $\mathcal{T}$ is lower semicontinuous (see, e.g., \cite{WYu}).\\
 For $t>0$,  set
 $$\mathcal{R}(t): = \{x\in \R^n: \mathcal{T}(x) \le t\},$$
 the \textit{controllable set} is the set $$\mathcal{R}: = \{x\in\R^n: \mathcal{T}(x) <+\infty\} = \bigcup_{t\ge0}\mathcal{R}(t)$$
and  the \textit{attainable set} from $\mathcal{K}$ at time $t$ is the set
 $$\mathcal{A}(t): = \{x(t): x(\cdot) \,\,\text{solves}\,\,(\ref{DI})\,\,\text{with}\,\,x(0)\in \mathcal{K}\}.$$
The set $\mathcal{A}(t)$ is also referred to as the \textit{reachable set}, or the \textit{accessibility set}, from $\mathcal{K}$ at time $t$. It is well-known that, under assumption (F), $\mathcal{R}(t)$ and $\mathcal{A}(t)$ are compact for every $t$ (see, e.g.,\cite{AC}). 
\section{Variational analysis results}\label{SectV}
This section is devoted to the variational analysis of the minimum time function for differential inclusion under assumption (F) only. Recall that the minimized Hamiltonian associated to $F$ is the function $h:\R^n\times \R^n \to \R$ defined by
\begin{equation}
\label{DefMH}
h(x,\zeta) = \min_{v\in F(x)}\langle v,\zeta\rangle \qquad \forall x,\zeta \in \R^n.
\end{equation}
In \cite{WYu}, the authors proved the following interesting characterizations of the proximal subdifferential of the minimum time function $\mathcal{T}$ at points inside the target as well as outside the target.
\begin{Theorem}\cite{WYu} \label{WY}
Assume that the multifunction $F$ satisfies assumption (F).
\begin{itemize}
\item[(a)] For all $x\in \mathcal{K}$, we have
$$\partial^P\mathcal{T}(x) = N^P_\mathcal{K}(x) \cap \{\zeta\in \R^n: h(x,\zeta) \ge -1\}.$$
\item[(b)] Whenever $r>0$ and $\mathcal{T}(x) = r$, then we have
$$\partial^P\mathcal{T}(x) = N^P_{\mathcal{R}(r)}(x) \cap \{\zeta\in \R^n: h(x,\zeta) = -1\}.$$
\end{itemize}
\end{Theorem}
The next result is the first main result of the current paper which is similar to the result in Theorem \ref{WY}, but is proven for the proximal horizontal sudifferential.
 \begin{Theorem} \label{THM1}
 Assume that the multifunction $F$ satisfies assumption (F). 
\begin{itemize}
\item[(a)] Let $x_0\in \mathcal{K}$. We have
 \begin{equation}
 \label{Eq31}
 \partial ^{\infty} \mathcal{T}(x_0) = N_{\mathcal{K}}^P (x_0) \cap \{\zeta \in \R^n: h(x_0,\zeta) \ge 0\}.
 \end{equation}
 \item[(b)] Let $x_0 \in \mathcal{R}\setminus \mathcal{K}$. We have
 \begin{equation}
 \partial ^\infty \mathcal{T}(x_0) = N_{\mathcal{R}(\mathcal{T}(x_0))}^P (x_0) \cap \left\{\zeta\in\R^n: h(x_0,\zeta) = 0\right\}.
 \end{equation}
\end{itemize} 
 \end{Theorem}
 Before beginning the proof of Theorem \ref{THM1}, we prove the following lemma.
 
 \begin{Lemma} \label{THM2}
 Assume  (F). Let $x_0 \in \mathcal{R}\setminus \mathcal{K}$ and $\zeta \in N^P_{\mathcal{R}(\mathcal{T}(x_0))}(x_0)$. One has $h(x_0,\zeta)\le 0$.
 \end{Lemma}
 \begin{proof}
    Since $\zeta \in N^P_{\mathcal{R}(\mathcal{T}(x_0))}(x_0)$, there exists $\sigma >0$ such that
    \begin{equation}
    \label{Eq35}
    \langle \zeta, y-x_0\rangle \le \sigma |y-x_0|^2,
    \end{equation}
  for all $y \in \mathcal{R}(\mathcal{T}(x_0))$.
  \\Let $x(\cdot)$ be an optimal trajectory for $x_0$. Then $x(t) \in \mathcal{R}(\mathcal{T}(x_0))$ for all $t\in [0,\mathcal{T}(x_0)]$. Let $y(\cdot)$ be the measurable function which is the projection of $\dot{x}(\cdot)$ on $F(x_0)$ restricted to $[0,\mathcal{T}(x_0)]$. By Gronwall's Lemma and by the Lipschitzianity of $F$, we have
  $$|\dot{x}(t) - y(t)| \le L|x(t) - x_0| \le LMt,\,\,\,\,\,\text{for a.e.}\, t\in [0,\mathcal{T}(x_0)].$$
  For $t\in (0,\mathcal{T}(x))$, taking $y: = x(t)$ in (\ref{Eq35}), we have
  $$\langle \zeta, x(t) - x_0\rangle \le \sigma |x(t) - x_0|^2,$$
  or, equivalently,
  $$\langle \zeta, \int_0^t \dot{x}(s)ds \rangle \le \sigma Mt^2.$$
  We have, for $t\in (0,\mathcal{T}(x_0))$,
  \begin{eqnarray*}
  h(x_0,\zeta)t &\le& \int_0^t \langle \zeta, y(s)\rangle ds\le \sigma M t^2 + \int_0^t \langle \zeta, y(s) - \dot{x}(s)\rangle ds \\
  &\le& \sigma Mt^2 + ML|\zeta|\int_0^t tds \le \sigma Mt^2 + ML|\zeta|t^2
  \end{eqnarray*}
  This implies that $h(x_0,\zeta) \le 0$. The proof is complete.
 \end{proof}
 
 We are now ready to prove Theorem \ref{THM1}.
 
 \noindent \textit{Proof of Theorem \ref{THM1}}. (a) Let $\zeta \in \partial ^\infty \mathcal{T}(x_0)$. Then $(\zeta,0) \in N^P_{\mathrm{epi}(\mathcal{T})}(x_0,\mathcal{T}(x_0)$. Thus there exist $\sigma >0, \eta >0$ such that
 \begin{equation}
 \label{Eq32}
 \langle \zeta, y - x_0\rangle \le \sigma \left(|y-x_0|^2 + \beta^2\right),
 \end{equation}
 for all $y \in B(x_0,\eta)$ and $\beta \ge \mathcal{T}(y)$.
\\Taking $y\in B(x_0,\eta) \cap \mathcal{K}$ and $\beta = \mathcal{T}(y) = 0$ in (\ref{Eq32}), we have
$$\langle \zeta,y-x_0\rangle \le \sigma |y-x_0|^2.$$
It follows that $\zeta \in N_{\mathcal{K}}^P (x_0)$. \\
We are now going to show that $h(x_0,\zeta) \ge 0$. Let $w\in F(x_0)$ be such that
$$\langle \zeta,w\rangle = h(x_0,\zeta) = \min_{v\in F(x_0)} \langle \zeta,v\rangle.$$
By Theorem \ref{C1}, there exists a $C^1$ trajectory $y(\cdot)$ on $[0,T]$, for some $T>0$, of $-F$ satisfying $y(0) = x_0$ and $\dot{y}(0) = -w$. By Gronwall's Lemma, there is some constant $M>0$ such that $|y(t) - x_0| \le Mt$ for all $t \in [0,T]$. \\
There are two possible cases.\\
\textbf{Case 1}. There exists $\varepsilon >0$ such that $y(t) \in \mathcal{K}\cap B(x_0,\eta)$ for all $t\in [0,\varepsilon]$. For $t\in (0,\varepsilon)$, taking $y:=y(t)$ and $\beta := \mathcal{T}(y(t)) = 0$ in (\ref{Eq32}), we have
$$\langle \zeta,y(t)-x_0\rangle \le \sigma |y(t) - x_0|^2 \le \sigma Mt^2,$$
or, equivalently,
$$\left\langle \zeta,\frac{y(t)-x_0}{t} \right\rangle\le \sigma Mt.$$
Letting $t\to 0+$ in the latter inequality and using the fact that $y(\cdot)$ is of class $C^1$ with $\dot{y}(0) = -w$, we get $\langle\zeta,-w\rangle \le 0$. Therefore, $h(x_0,\zeta) = \langle\zeta,w\rangle \ge 0$.\\
\textbf{Case 2}. There exists $\varepsilon >0$ such that $y(t)\not\in \mathcal{K}$ for all $t\in (0,\varepsilon]$. Fix $t\in (0,\varepsilon)$ such that $y(s)\in B(x_0,\eta)$ for all $s\in [0,t]$. Set $x(s) = y(t-s), s\in [0,t]$. Then $x(\cdot)$ is a trajectory of $F$ with $x(t) = x_0$. By the principle of optimality, we have 
$$\mathcal{T}(y(t)) = \mathcal{T}(x_0)\le t.$$
Taking $y: = y(t), \beta:= t\ge \mathcal{T}(y(t))$ in (\ref{Eq32}), we have
$$\langle \zeta, y(t) - x_0\rangle \le \sigma \left(|y(t)-x_0|^2 +t^2\right) \le \sigma(M+1)t^2.$$
or, equivalently,
$$\left\langle \zeta, \frac{y(t)-x_0}{t}\right\rangle \le \sigma(M+1)t.$$
Letting $t\to 0+$ in the latter inequality and using the fact that $y(\cdot)$ is of class $C^1$ with $\dot{y}(0) = -w$, we get $\langle\zeta,-w\rangle \le 0$. Therefore,  $h(x_0,\zeta) \ge 0$. 

Now let $\zeta \in N_{\mathcal{K}}^P(x_0)$ be such that $h(x_0,\zeta) \ge 0$. We are going to show that $\zeta \in \partial ^\infty \mathcal{T}(x_0)$, i.e., there is some $\sigma >0$ such that
$$\langle \zeta, y-x_0\rangle \le \sigma \left(|y-x_0|^2 + \beta^2\right),$$
for all $(y,\beta) \in \mathrm{epi}(\mathcal{T}), y \in \mathrm{dom}(\mathcal{T})$.

Let $y \in \mathrm{dom}(\mathcal{T})$ be arbitrary. Set $T:= \mathcal{T}(y)$. Let $x(\cdot)$ be an optimal trajectory for $y$. Set $x_1 = x(T)$. Then $x_1\in \mathcal{K}$. By Gronwall's Lemma, we have, for each $t\in [0,T]$,
$$|x(t) - x_0| \le |x(t)-y| + |y-x_0| \le Mt + |y-x_0|.$$
Since $x_1 \in \mathcal{K}, \zeta \in N_{\mathcal{K}}^P(x_0)$, there is some $\sigma_1>0$ such that
\begin{equation}
\label{Eq33}
\langle \zeta,x_1-x_0\rangle \le \sigma_1 |x_1-x_0|^2 \le \sigma_1\left(MT + |y-x_0|\right)^2.
\end{equation}
Let $y(\cdot)$ be a measurable function which is the projection of $\dot{x}(\cdot)$ on the set $F(x_0)$ restricted to $[0,T]$, i.e., for all most $t\in [0,T]$,
$$y(t) = \mathrm{proj}_{F(x_0)}(\dot{x}(t)) \in F(x_0).$$
Since $F$ is locally Lipschitz,
\begin{equation}
\label{Eq34}
|y(t) - \dot{x}(t)| \le L|x(0) - x(t)| \le LMT + L|y -x_0|,\,\,\,\,\,\text{a.e.}\,t\in [0,T].
\end{equation}
Using (\ref{Eq33}) and (\ref{Eq34}), we have the following estimate
\begin{eqnarray*}
\langle \zeta, y-x_0\rangle &=& \langle\zeta,y-x_1\rangle + \langle\zeta,x_1 - x_0\rangle\\
&\le& -\langle \zeta, \int_0^T \dot{x}(t)dt \rangle + \sigma_1(MT + |y-x_0|)^2\\
&=& -\int_0^T \langle \zeta,y(t)\rangle dt + \int_0^T \langle \zeta, y(t)-\dot{x}(t)\rangle dt +  \sigma_1(MT + |y-x_0|)^2\\
&\le& -h(x_0,\zeta)T + \int_0^T  |\zeta| |y(t)-\dot{x}(t)| dt +  \sigma_1(MT + |y-x_0|)^2\\
&\le& |\zeta|\left(LMT^2 + L|y-x_0|T\right) + \sigma_1(MT + |y-x_0|)^2\\
&\le& \sigma \left(|y-x_0|^2 + T^2\right),\,\,\,\,\,\text{for some}\,\,\sigma >0\\
\end{eqnarray*}
Therefore $\langle  \zeta, y-x_0\rangle \le  \sigma \left(|y-x_0|^2 + \beta^2\right)$, for all $\beta \ge T = \mathcal{T}(y)$. The conclusion is $\zeta \in \partial^\infty \mathcal{T}(x_0)$, and ends the proof of part (a).
  
  (b) Let $\zeta \in \partial^\infty \mathcal{T}(x_0)$. Then there exists $\sigma >0$ such that
  \begin{equation}
  \label{Eq36}
  \langle \zeta, y-x_0\rangle \le \sigma \left( |y-x_0|^2 + |\beta - \mathcal{T}(x_0)|^2\right),
  \end{equation}
  for all $(y,\beta) \in \mathrm{epi}(\mathcal{T})$.\\
  From (\ref{Eq36}), one has
  $$\langle \zeta, y -x_0\rangle \le \sigma |y-x_0|^2,$$
  for all $y\in \mathcal{R}(\mathcal{T}(x_0))$, i.e., $\zeta \in N^P_{\mathcal{R}(\mathcal{T}(x_0))}(x_0)$.
  \\It follows from Lemma \ref{THM2} that $h(x_0,\zeta) \le 0$. We are going to show that $h(x_0,\zeta) \ge 0$. Let $w \in F(x_0)$ be such that 
  $$\langle w,\zeta\rangle = h(x_0,\zeta) = \min_{v\in F(x_0)} \langle v,\zeta\rangle.$$
  There exists a $C^1$ trajectory $x(\cdot)$ of $-F$ on $[0,T]$ for some $T>0$ such that $x(0) = x_0$ and $\dot{x}(0) = -w$. Since $x_0 \not\in \mathcal{K}$, there exists $\varepsilon >0$ such that $x(t) \not\in\mathcal{K}$ for all $t\in [0,\varepsilon]$. Fix $t\in (0,\varepsilon)$. For $s\in [0,t]$, we define $y(s) = x(t-s)$. Then $y(\cdot)$ is a trajectory of $F$. By the principle of optimality, we have
  $$\mathcal{T}(x(t)) = \mathcal{T}(y(0)) \le \mathcal{T}(y(t)) + t = \mathcal{T}(x_0) + t.$$
  Taking $y: = x(t), \beta: = \mathcal{T}(x_0) + t$ in (\ref{Eq36}), we get
  $$\langle \zeta,x(t) -x_0\rangle \le \sigma \left(|x(t) - x_0|^2 + t^2\right) \le \sigma (M+1)t^2,$$
  or, equivalently,
  $$\left\langle \zeta,\frac{x(t)-x(0)}{t}\right\rangle \le \sigma (M+1)t.$$
  Letting $t\to 0+$ in the both sides of the latter inequality, we obtain $\langle \zeta,-w\rangle = \langle \zeta,\dot{x}(0)\rangle \le 0$. Hence $h(x_0,\zeta) \ge 0$.
 
 Now let $\zeta \in N^P_{\mathcal{R}(\mathcal{T}(x_0))}(x_0)$ with $h(x_0,\zeta) = 0$. We will show that $\zeta \in \partial^\infty \mathcal{T}(x_0)$, i.e.,  there exists a constant $\sigma >0$ such that
 \begin{equation}
  \label{Eq37}
  \langle \zeta, y-x_0\rangle \le \sigma \left( |y-x_0|^2 + |\beta - \mathcal{T}(x_0)|^2\right),
  \end{equation}
  for all $(y,\beta) \in \mathrm{epi}(\mathcal{T})$.
  
  Let $y\in \mathrm{dom}(\mathcal{T})$ be arbitrary. We have two possible cases
  
  \textbf{Case 1}.  $\mathcal{T}(y) \le \mathcal{T}(x_0)$. Then $y\in \mathcal{R}(\mathcal{T}(x_0))$. Since $\zeta \in N^P_{\mathcal{R}(\mathcal{T}(x_0))}(x_0)$, there exists $\sigma_1 >0$ such that
  $$\langle\zeta, y-x_0\rangle \le \sigma_1 |y-x_0|^2,$$
  i.e., (\ref{Eq37}) holds for $\sigma = \sigma_1$.
  
 \textbf{ Case 2}. $\mathcal{T}(y) > \mathcal{T}(x_0)$. Let $y(\cdot)$ be an optimal trajectory for $y$. Then by the principle of optimality,
 $$\mathcal{T}(y) = \mathcal{T}(y(t)) + t\,\,\,\text{for all}\,\,t\in [0,\mathcal{T}(y)].$$
 Set $r = \mathcal{T}(y) - \mathcal{T}(x_0)$ and $x_1 = y(r)$. Then $\mathcal{T}(x_1) = \mathcal{T}(x_0)$ and thus $x_1 \in \mathcal{R}(\mathcal{T}(x_0))$. There is some $\sigma_1 >0$ such that
  $$\langle\zeta, y-x_0\rangle \le \sigma_1 |y-x_0|^2.$$
 By Gronwall's Lemma, we have, for $t\in [0,\mathcal{T}(y)]$,
 $$|y(t) -x_0| \le |y(t) -y| + |y-x_0| \le Mt + |y-x_0|.$$
 In particular, $|x_1 - x_0| \le Mr + |y - x_0|$.
 
 Let $z(\cdot)$ be the measurable function which is the projection of $\dot{y}(\cdot)$ on $F(x_0)$ restricted to $[0,\mathcal{T}(y)]$, i.e.,
 $$z(t) = \mathrm{proj}_{F(x_0)} \dot{y}(t),\,\,\,\,\text{for all most}\,\,t\in [0,\mathcal{T}(y)].$$
 By the Lipschitz continuity of $F$,
 \begin{equation}
 \label{Eq38}
 |\dot{y}(t) - z(t)| \le L|y(t) - x_0| \le LMt+ L|y-x_0|,\,\,\,\,\text{for all most}\,\,t\in [0,\mathcal{T}(y)].
 \end{equation}
 We have the estimate
 \begin{eqnarray*}
 \langle \zeta, y -x_0\rangle &=& \langle \zeta, y-x_1\rangle + \langle \zeta, x_1 - x_0\rangle \\
 &\le & - \langle\zeta,\int_0^r \dot{y}(t)dt\rangle + \sigma_1|x_1 - x_0|^2\\
 &=& - \langle \zeta,\int_0^rz(t)dt \rangle + \langle \zeta,\int_0^r (z(t) - \dot{y}(t))dt \rangle +  \sigma_1|x_1 - x_0|^2\\
 &\le& -h(x_0,\zeta)r + |\zeta|\int_0^r |z(t)-\dot{y}(t)|dt + \sigma_1|x_1 - x_0|^2\\
 &\le& |\zeta|\int_0^r(LMr +L|y-x_0|)dt + \sigma_1|x_1 - x_0|^2\\
 &\le& L|\zeta| (Mr^2 + |y-x_0|r) + \sigma_1(Mr + |y-x_0|)^2\\
 &\le& \sigma (|y-x_0|^2 + r^2) = \sigma (|y-x_0|^2 + |\mathcal{T}(y) - \mathcal{T}(x_0)|^2\\
 &\le& \sigma (|y-x_0|^2 + |\beta - \mathcal{T}(x_0)|^2
 \end{eqnarray*}
 for some $\sigma >0$ and for all $\beta \ge \mathcal{T}(y) >\mathcal{T}(x_0)$. This ends the proof. \qed
 
  \begin{Corollary}
  \label{Co1} Assume (F). We have 
  $$  N^P_{\mathcal{R}(\mathcal{T}(x))}(x)=\R_+ \partial^P \mathcal{T}(x) \cup \partial^\infty \mathcal{T}(x),$$
  for $x\in \mathcal{R}\setminus \mathcal{K}$. 
  \end{Corollary}
  \begin{proof}
  Let $\zeta \in N^P_{\mathcal{R}(\mathcal{T}(x))}(x)$. Then by Lemma \ref{THM2}, we have $h(x,\zeta) \le 0$. If $h(x,\zeta) = 0$, then by Theorem \ref{THM1}, $\zeta \in \partial^\infty \mathcal{T}(x)$. If $h(x,\zeta) <0$, then we set $\eta = -\zeta/h(x,\zeta)$. Observe that $\eta \in N^P_{\mathcal{R}(\mathcal{T}(x))}(x)$ and $h(x,\eta) = -1$. It follows from Theorem \ref{WY} that $\eta \in \partial^P\mathcal{T}(x)$. Thus $\zeta = -h(x,\zeta)\eta \in \R_+ \partial^P\mathcal{T}(x)$. Therefore $  N^P_{\mathcal{R}(\mathcal{T}(x))}(x)\subset \R_+ \partial^P \mathcal{T}(x) \cup \partial^\infty \mathcal{T}(x).$ The oposite inclusion follows easily from Theorem \ref{THM1}, Theorem \ref{WY} and the definition of a cone.
  \end{proof}
  The second main result of this section is a connection between normal cones to sublevel sets and to the epigraph of the minimum time function. This contains generalizations of the results in \cite{CKL, LTS}. 
  \begin{Theorem}
  \label{THM4}
  Assume  (F). Let $x\in \mathcal{R}\setminus \mathcal{K}$. 
  \begin{itemize}
  \item[(i)] if $\zeta \in N^P_{\mathcal{R}(\mathcal{T}(x))}(x)$, then $(\zeta,h(x,\zeta)) \in N^P_{\mathrm{epi}(\mathcal{T})}(x,\mathcal{T}(x))$.
  \item[(ii)] if $\zeta \in \R^n$ and $\alpha \in \R$ satisfy $(\zeta,\alpha) \in N^P_{\mathrm{epi}(\mathcal{T})}(x,\mathcal{T}(x))$, then $\alpha \le 0$, $\zeta \in N^P_{\mathcal{R}(\mathcal{T}(x))}(x)$ and $h(x,\zeta) = \alpha$.
  \end{itemize}
  \end{Theorem}
  \begin{proof}
  (i) Since $\zeta \in N^P_{\mathcal{R}(\mathcal{T}(x))}(x)$, it follows from Lemma \ref{THM2} that $h(x,\zeta) \le 0$. There are two possible cases
  
  (a) Case 1: $h(x,\zeta) = 0$. Then by Theorem \ref{THM1} $\zeta \in \partial^\infty \mathcal{T}(x)$, i.e., $(\zeta,h(x,\zeta)) = (\zeta,0) \in N^P_{\mathrm{epi}(\mathcal{T})}(x,\mathcal{T}(x))$.
  
  (b) Case 2: $h(x,\zeta) <0$. Set $\zeta_1 = -\frac{\zeta}{h(x,\zeta)}$. Observe that $\zeta_1 \in N^P_{\mathcal{R}(\mathcal{T}(x))}(x)$ and $h(x,\zeta_1) = -1$. It follows from Theorem \ref{WY} that $\zeta_1 \in \partial^P\mathcal{T}(x)$, i.e., $(\zeta_1,-1) \in N^P_{\mathrm{epi}(\mathcal{T})}(x,\mathcal{T}(x))$. Thus $(\zeta,h(x,\zeta))  = -h(x,\zeta) (\zeta_1,-1) \in N^P_{\mathrm{epi}(\mathcal{T})}(x,\mathcal{T}(x))$.
  
  (ii) Since $(\zeta,\alpha)  \in N^P_{\mathrm{epi}(\mathcal{T})}(x,\mathcal{T}(x))$, by the nature of an epigraph, one has $\alpha \le 0$. We also have two possible cases
  
  (a) Case 1: $\alpha = 0$. Then $(\zeta,0) \in  N^P_{\mathrm{epi}(\mathcal{T})}(x,\mathcal{T}(x))$, i.e., $\zeta \in \partial^\infty \mathcal{T}(x)$. Thanks to Theorem \ref{THM1}, $\zeta \in N^P_{\mathcal{R}(\mathcal{T}(x))}(x)$ and $h(x,\zeta) = 0 = \alpha$.
  
  (b) Case 2: $\alpha <0$. Set $\zeta_1 =- \frac{\zeta}{\alpha}$. Then $(\zeta_1,-1) = -\frac{1}{\alpha}(\zeta,\alpha) \in N^P_{\mathrm{epi}(\mathcal{T})}(x,\mathcal{T}(x))$, i.e., $(\zeta_1 \in \partial^P \mathcal{T}(x)$. It follows from Theorem \ref{WY} that $\zeta_1 \in N^P_{\mathcal{R}(\mathcal{T}(x))}(x)$ and $h(x,\zeta_1) = -1$. Therefore $\zeta  = - \alpha \zeta_1 \in N^P_{\mathcal{R}(\mathcal{T}(x))}(x)$ and $h(x,\zeta) = -\alpha h(x,\zeta_1) = \alpha$.
  \end{proof}
 \begin{Remark} We note that the statements (i) and (ii) in Theorem \ref{THM4} were proved in \cite{CKL,LTS} for the case $F$ given in the form 
 $$F(x) = \{f(x,u): u\in U\}, \qquad x\in \R^n,$$
  with $f:\R^n \times U \to \R^n$ of class $C^{1,1}$ with respect to the first argument, under very strong assumptions, namely there exists a neighborhood $\mathcal{W}$ of $x$ such that
  \begin{itemize}
  \item[(1)] $\mathcal{T}$ is finite and continuous in $\mathcal{W}$;
  \item[(2)] every $y\in \mathcal{W}$ is an optimal point, i.e., there exists an optimal trajectory which passes through $y$.
  \item[(3)] for every $y\in \mathcal{W}$, the optimal control steering $y$ to the target $\mathcal{K}=\{0\}$ is unique and bang - bang with finitely many switching,
  \item[(4)] there exist $r>0$  and a continuous function $\varphi$ such that $\mathcal{R}(t)$ is $\varphi$ - convex for all $t< r$.
  \end{itemize}
 Moreover, the proofs in \cite{CKL,LTS} are based on Maximum Principle which is not used anywhere in this section.
 \end{Remark}
\begin{Lemma}\label{Lem1} Assume (F). Let $x\in \mathcal{R}\setminus \mathcal{K}$. One has
$$N^P_{\mathcal{R}(\mathcal{T}(x))}(x) = \{0\}\,\,\,\text{if and only if}\,\,\,N^P_{\mathrm{epi}(\mathcal{T})}(x,\mathcal{T}(x)) = \{0\}.$$
\end{Lemma}
\begin{proof}
Suppose $N^P_{\mathcal{R}(\mathcal{T}(x))}(x) = \{0\}$. We will show that $N^P_{\mathrm{epi}(\mathcal{T})}(x,\mathcal{T}(x)) = \{0\}$. Assume, to the contrary, that $N^P_{\mathrm{epi}(\mathcal{T})}(x,\mathcal{T}(x)) \ne \{0\}$. Let $\zeta \in \R^n, \alpha \in \R$ be such that $(\zeta,\alpha) \in N^P_{\mathrm{epi}(\mathcal{T})}(x,\mathcal{T}(x))$ and $(\zeta,\alpha) \ne (0,0)$. From Theorem \ref{THM4}, we have $\zeta \in N^P_{\mathcal{R}(\mathcal{T}(x))}(x)$ and $h(x,\zeta) = \alpha$. Since $N^P_{\mathcal{R}(\mathcal{T}(x))}(x) = \{0\}$, we get $\zeta = 0$ and then $\alpha = h(x,\zeta) = 0$. This contradicts to $(\zeta,\alpha)\ne (0,0)$.

We now assume that $N^P_{\mathrm{epi}(\mathcal{T})}(x,\mathcal{T}(x)) = \{0\}$. Let $\zeta \in N^P_{\mathcal{R}(\mathcal{T}(x))}(x)$. Again from Theorem \ref{THM4}, one has $(\zeta,h(x,\zeta)) \in N^P_{\mathrm{epi}(\mathcal{T})}(x,\mathcal{T}(x)) = \{0\}$. This implies $\zeta = 0$. Therefore $N^P_{\mathcal{R}(\mathcal{T}(x))}(x) = \{0\}$.
\end{proof}
By using above results in the next theorem we will show that the normal cones to the sublevel set and to the epigraph of $\mathcal{T}$ have the same dimension. This is a special feature of the minimum time function which is not shared for general functions, not even for convex functions. This feature was proved in \cite{GCLN} in the case of normal linear control systems. We note that the proof in \cite{GCLN} is based on, among other things, an explicit representation of the minimized Hamiltonian for  linear control systems. Of course, we cannot compute explicitly the minimized Hamiltonian for more general setting, e.g, nonlinear setting and assuming merely assumption (F). 
\begin{Theorem}
\label{THM5}
Assume  (F). For any $x\in \mathcal{R}\setminus \mathcal{K}$, we have
\begin{equation}
\label{Eq39}
\dim N^P_{\mathcal{R}(\mathcal{T}(x))}(x) = \dim N^P_{\mathrm{epi}(\mathcal{T})}(x,\mathcal{T}(x)) .
\end{equation}
\end{Theorem}
\begin{proof}
By Lemma \ref{Lem1}, it is enough to show that (\ref{Eq39}) holds true when $N^P_{\mathcal{R}(\mathcal{T}(x))}(x) \ne \{0\}$ and $ N^P_{\mathrm{epi}(\mathcal{T})}(x,\mathcal{T}(x)) \ne \{0\}$. Assume that $\dim N^P_{\mathcal{R}(\mathcal{T}(x))}(x) = \kappa \ge 1$ and $\dim N^P_{\mathrm{epi}(\mathcal{T})}(x,\mathcal{T}(x)) = \ell\ge 1$. We now assume that $\zeta_1,\cdots,\zeta_\kappa \in  N^P_{\mathcal{R}(\mathcal{T}(x))}(x)$ and they are linearly independent. It follows from Theorem \ref{THM4} that $(\zeta_1,h(x,\zeta_1)),\cdots, (\zeta_\kappa,h(x,\zeta_\kappa)) \in N^P_{\mathrm{epi}(\mathcal{T})}(x,\mathcal{T}(x))$. One can see easily that $(\zeta_1,h(x,\zeta_1)),\cdots, (\zeta_\kappa,h(x,\zeta_\kappa))$ are linearly independent. Thus $\kappa \le \ell$.

Let us now assume that $(\eta_1,\alpha_1),\cdots, (\eta_\ell,\alpha_\ell) \in N^P_{\mathrm{epi}(\mathcal{T})}(x,\mathcal{T}(x))$ are linearly independent. Thanks to Theorem \ref{THM4}, $\eta_1,\cdots,\eta_\ell \in N^P_{\mathcal{R}(\mathcal{T}(x))}(x)$ and $h(x,\eta_i)=\alpha_i$ for all $ i = 1,\cdots,\ell$. Observe that $\eta_i \ne 0$ for all $i=1,\cdots,\ell$. Indeed, if $\eta_i = 0$ for some $i\in \{1,\cdots,\ell\}$ then $\alpha_i =h(x,\eta_i) =0$. Thus  $(\eta_1,\alpha_1),\cdots, (\eta_\ell,\alpha_\ell)$ are not linearly independent.  We are going to show that $\eta_1,\cdots,\eta_\ell$ are linearly independent. Assume, to the contrary, that there exists $a_1,\cdots, a_\ell \in \R$ such that  $a_1^2 + \cdots + a_\ell^2 \ne 0$ and 
\begin{equation}\label{Eq310}
\sum_{i=1}^\ell a_i\eta_i = 0.
\end{equation}
Set
$$ I = \{i: 1\le i \le \ell, a_i \ge 0\},\,\,\,\text{and}\,\,\,J = \{1,\cdots,\ell\} \setminus I.$$

If $I$ and $J$ are both nonempty, then (\ref{Eq310}) implies that
\begin{equation}
\label{Eq311}
\sum_{i\in I} a_i\eta_i = -\sum_{j\in J}a_j\eta_j.
\end{equation}
Since $a_i \ge0$ and $(\eta_i,\alpha) \in N^P_{\mathrm{epi}(\mathcal{T})}(x,\mathcal{T}(x))$ for all $i\in I$, we have 
$$\left(\sum_{i\in I}a_i\eta_i,\sum_{i\in I}a_i\alpha_i\right) = \sum_{i\in I}a_i\left(\eta_i,\alpha_i\right) \in N^P_{\mathrm{epi}(\mathcal{T})}(x,\mathcal{T}(x)).$$
It follows from Theorem \ref{THM4} that
$$h\left(x,\sum_{i\in I}a_i\eta_i\right) =\sum_{i\in I}a_i\alpha_i.$$
Similarly, there holds
$$h\left(x,-\sum_{j\in J}a_i\eta_j\right) =-\sum_{j\in J}a_j\alpha_j.$$
The last two equalities together with (\ref{Eq311}) claim that
$$\sum_{i=1}^\ell a_i \alpha_i= 0.$$
Since
$$\sum_{i=1}^\ell a_i\left(\eta_i,\alpha_i\right) = \left(\sum_{i=1}^\ell a_i\eta_i,  \sum_{i=1}^\ell a_i\alpha_i \right) = (0,0),$$
and $(\eta_1,\alpha_1), \cdots, (\eta_\ell,\alpha_\ell)$ are linearly indepentdent, we get $a_i = 0$ for all $i\in \{1,\cdots,\ell\}$. This is a contradiction.

Similarly, one gets a contradiction if either $J=\emptyset$ or $I = \emptyset$. The proof is complete.
 \end{proof}
 
 \section{Sensitivity relations} \label{SectS}
 In this section, we use the results obtained in Section \ref{SectV} to derive some sensitivity relations. To do that, besides assumption (F), we need to assume some assumptions on the maximized Hamiltonian $H:\R^n \times \R^n \to \R$ associated to $F$ which is defined as follows
 \begin{equation}
 \label{DefMxH}
 H(x,p) = \max_{v\in F(x)} \langle v,p\rangle,\quad\quad \forall x,p\in \R^n.
 \end{equation}

 \textbf{Assumption (H)}. For every $r>0$
\begin{itemize}
\item[(H1)] there exists $c\ge 0$ so that for every $p\in \mathbb{S}^{n-1}$, the mapping $x\mapsto H(x,p)$ is semiconvex with semiconvexity constant $c$;
\item[(H2)] $\nabla_pH(x,p)$ exists and is Lipschitz in $x$ on $B(0,r)$, uniformly for $p\in \R^n\setminus \{0\}$.
\end{itemize}
Assumptions (H) was introduced for the minimum time propblem in \cite{camapw}. The following are some consequences of assumptions (F) and (H).
\begin{Proposition} (see, e.g., \cite{PCPW, CFS15})
Assume (F) and (H). For  $0\ne p\in \R^n$ and $x\in\R^n$, one has
\begin{equation}
\label{Eq23}
\partial H(x,p) = \partial_x H(x,p) \times \partial_p H(x,p)
\end{equation}
and
\begin{equation}
\label{Eq24}
\nabla_pH(x,p) \in F(x),\qquad \langle \nabla_pH(x,p),p\rangle = H(x,p).
\end{equation}
\end{Proposition}

 \begin{Lemma}\cite{CaKh} \label{CK} Let $G: [0,T] \times \R^n \rightrightarrows \R^n$ be an upper semicontinuous multifunction. Assume $G(t,\cdot)$ satisfies assumption (F) uniformly in $t\in [0,T]$ and is such that for some $K_0>0$, 
 $$|v| \le K_0|p|,\qquad \forall v\in G(t,p),\quad \forall (t,p) \in [0,T] \times \R^n.$$
 Let $p(\cdot)$ be a solution of the differential inclusion
 \begin{equation}\label{EqLem}
\left\{
  \begin{array}{lcl}
    \dot{p}(t) & \in &G(t,p(t)), \qquad a.e.\,\,t\in [0,T] \\
    p(0)& = & p_0.
  \end{array}
  \right.
\end{equation}
Then,
$$e^{-K_0t}|p(0)| \le |p(t)| \le e^{K_0t}|p(0)|, \qquad \forall t\in [0,T].$$
Moreover, for all $0 \le t_1\le t_2\le T$,
$$e^{-K_0(t_2-t_1)}|p(t_2)| \le |p(t_1)| \le e^{K_0(t_2-t_1)}|p(t_2)|$$
and
$$|p(t_2)-p(t_1)| \le K_0e^{K_0(t_2-t_1)}(t_2-t_1)|p(t_2)|.$$
 \end{Lemma}
We recall Maximum Principle in the following form
\begin{Theorem} \label{THCMP}
Assume that (F) and (H). Let $x_0\in \mathcal{R}\setminus \mathcal{K}$. Suppose $x(\cdot)$ is an optimal trajectory starting at $x_0$. Then there exists an absolutely continuous arc $p:[0,\mathcal{T}(x_0)] \to \R^n$, never vanishing, such that
\begin{equation}\label{MP}
\left\{
  \begin{array}{lcl}
    \dot{x}(s) & = &\nabla_p H(x(s),p(s)), \\
    -\dot{p}(s)& \in & \partial_xH(x(s),p(s),
  \end{array}
  \right.\,\,\,\,\,\text{a.e.}\,\,s\in [0,\mathcal{T}(x_0)],
\end{equation}
and the transverality condition $$p(\mathcal{T}(x_0)) \in N^C_{\mathcal{K}}(x(\mathcal{T}(x_0)).$$
\end{Theorem}
\begin{proof}
From Theorem 3.5.4 in \cite{FCL} and (\ref{Eq23}).
\end{proof}
An absolutely continuous function $p(\cdot)$ satisfying the system (\ref{MP}) and the transversality condition is called a dual arc associated to the trajectory $x(\cdot)$. From (\ref{Eq24}), we have
\begin{equation}
H(x(t),p(t))  = \langle \dot{x}(t),p(t)\rangle,\quad\text{for a.e.}\,\,t\in [0,\mathcal{T}(x_0)].
\end{equation}
We remark that, under our assumptions, if $(x,p)$ solves the Hamiltonian inclusion
\begin{equation}\label{HI}
\left\{
  \begin{array}{lcl}
    \dot{x}(s) & = &\nabla_p H(x(s),p(s)), \\
    -\dot{p}(s)& \in & \partial_xH(x(s),p(s),
  \end{array}
  \right.\,\,\,\,\,\text{a.e.}\,\,s\in [0,T],
\end{equation}
then there are two possible cases:
\begin{itemize}
\item[(a)] either $p(s) \ne 0$ for all $s\in [0,T]$.
\item[(b)] or $p(s) =0$ for all $s\in [0,T]$.
\end{itemize}
Moreover, let $r>0$ be such that $x([0,T]) \subset B(0,r)$ and let $K = K(r)$ be a Lipschitz constant of $F$ on $B(0,r)$, we have
$$|\dot{p}(s)| \le K|p(s)|,\qquad a.e.\, s\in [0,T].$$
(see, e.g., \cite{PCPW} for detailed discussion).

Finally, we recall the following result which is useful in the sequel.
\begin{Lemma} [see, e.g, \cite{CaKh}]
Assume (F) and (H), and let $p(\cdot)$ be an absolutely continuous arc on $[0,T]$ with $p(t)\ne 0$ for all $t\in [0,T]$. Then for each $x\in \R^n$, the problem 
\begin{equation}
\left\{
  \begin{array}{lcl}
    \dot{x}(t) & = &\nabla_p H(x(t),p(t)), \\
    x(0)& = & x,
  \end{array}
  \right.\,\,\,\,\,\text{a.e.}\,\,t\in [0,T],
\end{equation}
has a unique solution.
\end{Lemma}
The next theorem is the main result of this section. It can be seen as the propagation of the normals to the epigraph of the minimum time function along optimal trajectories.
\begin{Theorem}
\label{THM6}
Assume (F) and (H) and given $x_0\in \mathcal{R}\setminus \mathcal{K}$. Let $\bar{x}:[0,\mathcal{T}(x_0)] \to \R^n$ be an optimal trajectory for $x_0$ and let $\bar{p}:[0,\mathcal{T}(x_0)]\to \R^n$ be an arc such that $(\bar{x},\bar{p})$ is a solution of the system
\begin{equation}
\left\{
  \begin{array}{lcl}
    \dot{x}(s) & = &\nabla_p H(x(s),p(s)), \\
    -\dot{p}(s)& \in & \partial_xH(x(s),p(s),
  \end{array}
  \right.\,\,\,\,\,\text{a.e.}\,\,s\in [0,\mathcal{T}(x_0)],
\end{equation}
satisfying $\bar{x}(0) = x_0$ and $\left(-\bar{p}(0),h(x_0,-\bar{p}(0))\right) \in N^P_{\mathrm{epi}(\mathcal{T})}(x_0,\mathcal{T}(x_0)).$ Then for all $t\in [0,\mathcal{T}(x_0))$,
$$\left(-\bar{p}(t),h(\bar{x}(t),-\bar{p}(t))\right) \in N^P_{\mathrm{epi}(\mathcal{T})}(\bar{x}(t),\mathcal{T}(\bar{x}(t))).$$
Moreover, $h(\bar{x}(t),-\bar{p}(t)) = h(x_0,-\bar{p}(0))$ for all $t\in [0,\mathcal{T}(x_0)]$.
\end{Theorem}
 \begin{proof}
 We first note that if $\bar{p}(t) = 0$ for all $t\in [0, \mathcal{T}(x_0)]$ then the conclusion is trivial. We now suppose $\bar{p}(t)\ne 0$ for all $t\in [0, \mathcal{T}(x_0)]$. Set $\alpha = h(x_0,-\bar{p}(0))$. Since $(-\bar{p}(0),\alpha) \in N^P_{\mathrm{epi}(\mathcal{T})}(x_0,\mathcal{T}(x_0))$, there exist $C>0$ and $\eta >0$ such that
 \begin{equation}
 \label{eq41}
 \langle -\bar{p}(0),y-x_0\rangle + \alpha(\beta - \mathcal{T}(x_0)) \le C\left(|y-x_0|^2 + |\beta -  \mathcal{T}(x_0)|^2\right),
 \end{equation}
 for all $(y,\beta) \in \mathrm{epi}(\mathcal{T})$ with $y\in B(x_0,\eta)$.
 We fix $t\in (0,\mathcal{T}(x_0))$. Note that $\bar{x}(\cdot)$ is the unique solution of the system
 \begin{equation}
 \left\{
  \begin{array}{lcl}
    \dot{x}(s) & = &\nabla_p H(x(s),\bar{p}(s)), \\
    x(t)& = & \bar{x}(t),
  \end{array}
  \right.\,\,\,\,\,\text{for}\,\,s\in [0,t].
 \end{equation}
 For $h\in B(0,\eta)$, let $x_h:[0,t]\to \R^n$ be the solution of the equation
 \begin{equation}
 \left\{
  \begin{array}{lcl}
    \dot{x}(s) & = &\nabla_p H(x(s),\bar{p}(s)), \\
    x(t)& = & \bar{x}(t) + h,
  \end{array}
  \right.\,\,\,\,\,\text{for}\,\,s\in [0,t].
 \end{equation}
 Then by using Gronwall's Lemma, one can show that there exists $\kappa >0$ independent of $t$ such that
 \begin{equation}
 |x_h(s) - \bar{x}(s)| \le \kappa|h|,\,\,\,\,\text{for all}\,\,s\in [0,t].
 \end{equation}
 We can choose $\eta>0$ sufficiently small such that $x_h([0,t]) \cap \mathcal{K} = \emptyset$ for all $h\in B(0,\eta)$. By the principle of optimality, 
 $$\mathcal{T}(x_0) = \mathcal{T}(\bar{x}(t)) + t,$$
 and
 $$\mathcal{T}(x_h(0)) \le \mathcal{T}(x_h(t)) + t.$$
 For $\bar{\beta} \ge \mathcal{T}(x_h(t))$, we have
 $$\mathcal{T}(x_h(0)) \le \bar{\beta} + \mathcal{T}(x_0) - \mathcal{T}(\bar{x}(t)).$$
 In (\ref{eq41}), taking $y: = x_h(0),\beta := \bar{\beta} + \mathcal{T}(x_0) - \mathcal{T}(\bar{x}(t))$, we obtain
 \begin{equation}\label{eq42}
 \langle - \bar{p}(0),x_h(0) - \bar{x}(0)\rangle \le -\alpha (\bar{\beta} - \mathcal{T}(\bar{x}(t))) + C\left(|x_h(0) - x(0)|^2 + |\bar{\beta} - \mathcal{T}(\bar{x}(t))|^2\right).
 \end{equation}
 It follows that
 \begin{eqnarray}\label{eq43}
 \langle -\bar{p}(t),x_h(t) -\bar{x}(t)\rangle &=& \langle -\bar{p}(t),x_h(t) -\bar{x}(t) \rangle + \langle \bar{p}(0),x_h(0) - x(0)\rangle +  \langle - \bar{p}(0),x_h(0) - x(0)\rangle \nonumber\\ 
 &\le& \langle -\bar{p}(t),x_h(t) -\bar{x}(t) \rangle + \langle \bar{p}(0),x_h(0) - x(0)\rangle\nonumber \\
&& -\alpha (\bar{\beta} - \mathcal{T}(\bar{x}(t))) + C\left(|x_h(0) - x(0)|^2 + |\bar{\beta} - \mathcal{T}(\bar{x}(t))|^2\right).
 \end{eqnarray}
 We have
 \begin{eqnarray}\label{eq44}
&& \langle -\bar{p}(t),x_h(t) -\bar{x}(t) \rangle + \langle \bar{p}(0),x_h(0) - x(0)\rangle = \int_0^t \frac{d}{ds} \langle -\bar{p}(s),x_h(s) -\bar{x}(s)\rangle ds \nonumber\\
&=& \int_0^t \left( \langle -\dot{\bar{p}}(s),x_h(s) - \bar{x}(s)\rangle  + \langle-\bar{p}(s), \dot{x}_h(s) - \dot{\bar{x}}(s)\rangle \right) ds \nonumber\\
&=& \int_0^t\left( -\langle \dot{\bar{p}}(s),x_h(s) - \bar{x}(s)\rangle  - H(x_h(s),\bar{p}(s)) + H(\bar{x}(s),\bar{p}(s))\right) ds.
 \end{eqnarray}
 Since $ -\dot{\bar{p}}(s) \in \partial_xH(\bar{x}(s),\bar{p}(s))$ a.e. in $[0,\mathcal{T}(x_0)]$, it follows from assumption (H1) and Proposition \ref{ProSC} that
 \begin{eqnarray}
 \label{eq45}
 \langle -\bar{p}(t),x_h(t) -\bar{x}(t) \rangle + \langle \bar{p}(0),x_h(0) - x(0)\rangle &\le& C_1\int_0^t |\bar{p}(s)||x_h(s) - \bar{x}(s)|^2 ds \nonumber\\
 & \le& C_2 |h|^2 = C_2|x_h(t) - \bar{x}(t)|^2,
 \end{eqnarray}
 where $C1 >0, C_2>0$ are suitable constants independent of $t$.\\
 From (\ref{eq43}) - (\ref{eq45}), we have for all $h\in B(0,\eta) , \bar{\beta} \ge \mathcal{T}(x_h(t))$,
 \begin{equation}
  \langle -\bar{p}(t),x_h(t) -\bar{x}(t)\rangle  + \alpha (\bar{\beta} - \mathcal{T}(\bar{x}(t))) \le K(|x_h(t)-\bar{x}(t)|^2 + |\bar{\beta} - \mathcal{T}(\bar{x}(t))|^2,
 \end{equation}
 where $K>0$ is a suitable constant. This implies that $ (-\bar{p}(t),\alpha) \in N^P_{\mathrm{epi}(\mathcal{T})}(\bar{x}(t),\mathcal{T}(\bar{x}(t))).$ Thanks to Theorem \ref{THM4}, $\alpha = h(\bar{x}(t),-\bar{p}(t))$. Since $t\in (0,\mathcal{T}(x_0))$ is arbitrary, we have
 $$(-\bar{p}(t),h(\bar{x}(t),-\bar{p}(t))) \in N^P_{\mathrm{epi}(\mathcal{T})}(\bar{x}(t),\mathcal{T}(\bar{x}(t)))\,\,\text{and}\,\,h(\bar{x}(t),-\bar{p}(t)) = h(x_0,-\bar{p}(0)),\,\,\,\text{for all}\,\,t\in [0,\mathcal{T}(x_0)).$$
 Moreover, by continuity, we have
 $$h(\bar{x}(t),-\bar{p}(t)) = h(x_0,-\bar{p}(0)),\,\,\,\text{for all}\,\,t\in [0,\mathcal{T}(x_0)].$$
 The proof is complete.
 \end{proof}
 \begin{Theorem}
 \label{THM7}
 Assume (F) and (H) and given $x_0\in \mathcal{R}\setminus \mathcal{K}$. Let $\bar{x}:[0,\mathcal{T}(x_0)] \to \R^n$ be an optimal trajectory for $x_0$ and let $\bar{p}:[0,\mathcal{T}(x_0)]\to \R^n$ be an arc such that $(\bar{x},\bar{p})$ is a solution of the system
\begin{equation}
\left\{
  \begin{array}{lcl}
    \dot{x}(s) & = &\nabla_p H(x(s),p(s)), \\
    -\dot{p}(s)& \in & \partial_xH(x(s),p(s),
  \end{array}
  \right.\,\,\,\,\,\text{a.e.}\,\,s\in [0,\mathcal{T}(x_0)],
\end{equation}
satisfying $\bar{x}(0) = x_0$ and $-\bar{p}(0) \in N^P_{\mathcal{R}(\mathcal{T}(x_0))}(x_0)$. Then, for all $t\in [0,\mathcal{T}(x_0)]$,
$$-\bar{p}(t) \in N^P_{\mathcal{R}(\mathcal{T}(\bar{x}(t)))}(\bar{x}(t)) \quad \text{and}\quad h(\bar{x}(t),-\bar{p}(t)) = h(x_0,-\bar{p}(0)).$$
 \end{Theorem}
 \begin{proof}
Since  $-\bar{p}(0) \in N^P_{\mathcal{R}(\mathcal{T}(x_0))}(x_0)$, by Theorem \ref{THM4}, we have $$\left(-\bar{p}(0),h(x_0,-\bar{p}(0))\right) \in N^P_{\mathrm{epi}(\mathcal{T})}(x_0,\mathcal{T}(x_0)).$$
 By Theorem \ref{THM6}, $\left(-\bar{p}(t),h(\bar{x}(t),-\bar{p}(t))\right) \in N^P_{\mathrm{epi}(\mathcal{T})}(\bar{x}(t),\mathcal{T}(\bar{x}(t)))$, for all $t\in [0,\mathcal{T}(x_0))$. Hence again by Theorem \ref{THM4}, $-\bar{p}(t) \in N^P_{\mathcal{R}(\mathcal{T}(\bar{x}(t)))}(\bar{x}(t))$, for all $t\in [0,\mathcal{T}(x_0))$. Set $T = \mathcal{T}(x_0)$. In order to finish the proof, we only have to show that $-\bar{p}(T) \in N^P_\mathcal{K}(\bar{x}(T))$. The arguments follow the lines of the proof of Theorem \ref{THM6}.
 
 Since $-\bar{p}(0) \in N^P_{\mathcal{R}(T)}(x_0)$, there exist $C_0 >0$ and $\eta_0 >0$ such that 
 \begin{equation}
 \langle -\bar{p}(0),y_0 - x_0\rangle \le C_0|y_0-x_0|^2,
 \end{equation}
 for all $y_0\in \mathcal{R}(T) \cap B(x_0,\eta_0)$.\\
 Now let $y\in \mathcal{K}\cap B(\bar{x}(T),\eta_0)$ and set $h: = y-\bar{x}(T) \in B(0,\eta_0)$. Let $x_h:[0,T] \to \R^n$ be the solution of the system
 \begin{equation}
\left\{
  \begin{array}{lcl}
    \dot{x}(s) & = &\nabla_p H(x(s),\bar{p}(s)), \\
    x(T)& = & \bar{x}(T) + h,
  \end{array}
  \right.\,\,\,\,\,\text{for}\,\,s\in [0,T].
\end{equation}
Recall that $\bar{x}(\cdot)$ is the solution of the system
 \begin{equation}
\left\{
  \begin{array}{lcl}
    \dot{x}(s) & = &\nabla_p H(x(s),\bar{p}(s)), \\
    x(T)& = & \bar{x}(T),
  \end{array}
  \right.\,\,\,\,\,\text{for}\,\,s\in [0,T].
\end{equation}
Then there exists a constant $K_0 >0$ such that
\begin{equation}
\label{eq46}
|x_h(s) - \bar{x}(s)| \le K_0|h|,\,\,\,\forall\,\,s\in [0,T].
\end{equation}
Since $x_h(T) = y\in \mathcal{K}$, $\mathcal{T}(x_h(0)) \le T$. That is $x_h(0) \in \mathcal{R}(T)$. Thanks to (\ref{eq46}), $$x_h(0) \in \mathcal{R}(T) \cap B(x_0,\eta_0).$$
Thus 
\begin{equation}
\langle -\bar{p}(0), x_h(0) - \bar{x}(0)\rangle \le C_0|x_h(0) - \bar{x}(0)|^2.
\end{equation}
Moreover, arguing as in the proof of Theorem \ref{THM6}, we have, for some constant $K_1>0$,
\begin{equation}
\langle -\bar{p}(T),x_h(T) - \bar{x}(T)\rangle + \langle \bar{p}(0), x_h(0) - \bar{x}(0)\rangle \le K_1|x_h(T) - \bar{x}(T)|^2.
\end{equation}
Therefore,
\begin{eqnarray}
\langle -\bar{p}(T),x_h(T) - \bar{x}(T) \rangle &=& \langle -\bar{p}(T),x_h(T) - \bar{x}(T)\rangle + \langle \bar{p}(0), x_h(0) - \bar{x}(0)\rangle + \langle -\bar{p}(0), x_h(0) - \bar{x}(0)\rangle\nonumber\\
&\le& C|x_h(T) - \bar{x}(T)|^2. 
\end{eqnarray}
Since $x_h(T) \in B(\bar{x}(T),\eta_0)\cap \mathcal{K}$, the latter inequality implies that $-\bar{p}(T) \in N^P_\mathcal{K}(\bar{x}(T))$. This ends the proof.
 \end{proof}
 \begin{Corollary}
  \label{Co2}
 Assume (F) and (H) and given $x_0\in \mathcal{R}\setminus \mathcal{K}$. Let $\bar{x}:[0,\mathcal{T}(x_0)] \to \R^n$ be an optimal trajectory for $x_0$ and let $\bar{p}:[0,\mathcal{T}(x_0)]\to \R^n$ be an arc such that $(\bar{x},\bar{p})$ is a solution of the system
\begin{equation}
\left\{
  \begin{array}{lcl}
    \dot{x}(s) & = &\nabla_p H(x(s),p(s)), \\
    -\dot{p}(s)& \in & \partial_xH(x(s),p(s),
  \end{array}
  \right.\,\,\,\,\,\text{a.e.}\,\,s\in [0,\mathcal{T}(x_0)],
\end{equation}
satisfying $\bar{x}(0) = x_0$ and $-\bar{p}(0) \in \partial^P\mathcal{T}(x_0)$. Then, for all $t\in [0,\mathcal{T}(x_0)]$,
$$-\bar{p}(t) \in \partial^P\mathcal{T}(\bar{x}(t)),\quad\text{and}\quad h(\bar{x}(t),-\bar{p}(t)) = -1.$$
 \end{Corollary}
 \begin{proof}
 Since $-\bar{p}(0) \in \partial^P\mathcal{T}(x_0)$,  $(-\bar{p}(0),-1)\in N^P_{\mathrm{epi}(\mathcal{T})}(x_0,\mathcal{T}(x_0))$ and $h(x_0,-\bar{p}(0)) = -1$. By Theorem \ref{THM6}, we have $\left(-\bar{p}(t),-1\right) \in N^P_{\mathrm{epi}(\mathcal{T})}(\bar{x}(t),\mathcal{T}(\bar{x}(t)))$ for all $t\in [0,\mathcal{T}(x_0))$. That is, $-\bar{p}(t) \in \partial^P \mathcal{T}(\bar{x}(t))$ for all $t\in [0,\mathcal{T}(x_0))$. Moreover, since $-\bar{p}(0) \in N^P_{\mathcal{R}(\mathcal{T}(x_0))}(x_0)$, by Theorem \ref{THM7}, we have $-\bar{p}(T) \in N^P_{\mathcal{K}}(\bar{x}(T))$. Together with $h(\bar{x}(T),-\bar{p}(T)) = -1$,  by Theorem \ref{WY}, one has $-\bar{p}(T)\in \partial^P\mathcal{T}(\bar{x}(T))$.
 \end{proof}
 Similarly, one has
 \begin{Corollary}
  \label{Co3}
 Assume (F) and (H) and given $x_0\in \mathcal{R}\setminus \mathcal{K}$. Let $\bar{x}:[0,\mathcal{T}(x_0)] \to \R^n$ be an optimal trajectory for $x_0$ and let $\bar{p}:[0,\mathcal{T}(x_0)]\to \R^n$ be an arc such that $(\bar{x},\bar{p})$ is a solution of the system
\begin{equation}
\left\{
  \begin{array}{lcl}
    \dot{x}(s) & = &\nabla_p H(x(s),p(s)), \\
    -\dot{p}(s)& \in & \partial_xH(x(s),p(s),
  \end{array}
  \right.\,\,\,\,\,\text{a.e.}\,\,s\in [0,\mathcal{T}(x_0)],
\end{equation}
satisfying $\bar{x}(0) = x_0$ and $-\bar{p}(0) \in \partial^\infty\mathcal{T}(x_0)$. Then, for all $t\in [0,\mathcal{T}(x_0)]$,
$$-\bar{p}(t) \in \partial^\infty\mathcal{T}(\bar{x}(t)),\quad\text{and}\quad h(\bar{x}(t),-\bar{p}(t)) = 0.$$
 \end{Corollary}
 \section{Regularity of the minimum time function}\label{SectR}
 In this section, we apply results in Section \ref{SectV} and Section \ref{SectS} to study the regularity of the minimum time function.
 
 For $\delta >0$, set $\mathcal{S}(\delta) = \mathcal{R}(\delta)\setminus \mathcal{K}$. For a subset  $\mathcal{O}$ of $\R^n$, let $\mathcal{T}_{|\mathcal{O}}:\mathcal{O} \to \R$ be the restriction of $\mathcal{T}$ on $\mathcal{O}$, i.e., $\mathcal{T}_{|\mathcal{O}}(x) = \mathcal{T}(x)$ for $x\in \mathcal{O}$.
 
\textbf{ Assumption (Q)}.  There exist constants $\delta >0$ and $\varphi_0 \ge 0$ such that $\mathcal{R}(t)$ is $\varphi_0$-convex for all $t\in [0,\delta]$.

In the following proposition, we present a relationship between regularity properties of sublevel sets and of the epigraph of the minimum time function.
\begin{Proposition} \label{ProEpi}
Assume (F), (H) and (Q). If  $\mathcal{T}$ is continuous in $\mathcal{R}$, then there exists a continuous function $\varphi$ such that the epigraph of $\mathcal{T}_{|\mathcal{S}(\delta)}$ is $\varphi$-convex.
\end{Proposition}
 \begin{proof}
 We first prove that there exists a constant $C = C(\delta,\varphi_0)$ such that for all $x,y \in \mathcal{S}(\delta)$ and $(\zeta,\alpha) \in N^P_{\mathrm{epi}(\mathcal{T})}(x,\mathcal{T}(x))$, there holds
 \begin{equation}
 \label{Eq51}
 \langle (\zeta,\alpha), (y,\mathcal{T}(y)) - (x,\mathcal{T}(x))\rangle \le C(|\zeta| + |\alpha|)\left( |y-x|^2 + |\mathcal{T}(y)-\mathcal{T}(x)|^2\right).
 \end{equation}
Since $(\zeta,\alpha) \in N^P_{\mathrm{epi}(\mathcal{T})}(x,\mathcal{T}(x))$, by Theorem \ref{THM4} we have that $\zeta \in N^P_{\mathcal{R}(\mathcal{T}(x))}(x)$ and $h(x,\zeta) = \alpha$. Hence if $\zeta = 0$ then $\alpha = 0$. Thus (\ref{Eq51}) holds. Now assume that $\zeta \ne 0$. We have two possible cases:
\begin{itemize}
\item[(i)] $\mathcal{T}(y)\ge \mathcal{T}(x)$,
\item[(ii)] $\mathcal{T}(y) < \mathcal{T}(x)$.
\end{itemize}
 We now deal with the case (i). Let $\bar{y}(\cdot)$ be an optimal trajectory starting from $y$. Set $r: = \mathcal{T}(y) - \mathcal{T}(x)$ and $y_1:=\bar{y}(r)$. Observe that $y_1 \in \mathcal{R}(\mathcal{T}(x))$. By Gronwall's Lemma, there exists a constant $K_1: = K_1(\delta)$ such that $|y - \bar{y}(t)| \le K_1t$ for all $t\in [0,\mathcal{T}(y)]$. We have
 \begin{equation}
 \label{Eq52}
 \langle \zeta, y -x\rangle = \langle \zeta, y_1-x\rangle + \langle \zeta, y - x_1\rangle = : (I) + (II).
 \end{equation}
 Since $y_1 \in \mathcal{R}(\mathcal{T}(x))$ and $\mathcal{R}(\mathcal{T}(x))$ is $\varphi_0$-convex, one has
 \begin{eqnarray}
 \label{Eq53}
 (I) &\le & \varphi_0|\zeta||y_1 - x|^2 \le \varphi_0|\zeta| (|y-x| + |y-y_1|)^2\nonumber\\
 &\le& \varphi_0|\zeta| (|y-x| + K_1r)^2\nonumber\\
 &\le& K_2|\zeta|\left(|y-x|^2 + |\mathcal{T}(y)-\mathcal{T}(x)|^2\right)
 \end{eqnarray}
 for some suitable constant $K_2: = K_2(\delta,\varphi_0)$.
 
 Now let $z(\cdot)$ be the measure function which is the projection of $\dot{\bar{y}}(\cdot)$ on $F(x)$ restricted to $[0,\mathcal{T}(y)]$, i.e.,
 $$z(t) = \mathrm{proj}_{F(x)}\dot{\bar{y}}(t),\,\,\text{for a.e.}\,\,t\in [0,\mathcal{T}(y)].$$
 By the  properties of $F$, there exists a constant $L: = L(\delta)$ such that
 $$|\dot{\bar{y}}(t) - z(t)| \le L|\bar{y}(t) - x| \le L(|\bar{y}(t)-y| + |y-x|) \le LK_1r + L|y-x|, \,\text{a.e.}\, t\in [0,r].$$
 Let us consider (II). We have
 \begin{eqnarray}
 \label{Eq54}
 (II) &=& -\int_0^r\langle \zeta,z(s)\rangle ds + \int_0^r\langle\zeta,z(s) - \dot{\bar{y}}(s)\rangle ds\nonumber\\
 &\le& - h(x,\zeta) r + |\zeta| \int_0^r \left(LK_1 r + L|y-x|)ds\right) \nonumber\\
 &\le& - h(x,\zeta)(\mathcal{T}(y) - \mathcal{T}(x)) + K_3|\zeta| \left (|y-x|^2 + |\mathcal{T}(y) - \mathcal{T}(x)|^2 \right)
 \end{eqnarray}
 for some suitable constant $K_3: = K_3(\delta)$.\\
 From (\ref{Eq52}) - (\ref{Eq54}) we obtain (\ref{Eq51}) for the case (i).
 
 We now consider the case (ii). Let $\bar{x}(\cdot)$ be an optimal trajectory starting from $x$. By Gronwall's Lemma, we may assume that $|\bar{x}(t) - x| \le K_1t$ for all $t\in [0,\mathcal{T}(x)]$. Let $\bar{p}(\cdot)$ be an arc such that $(\bar{x},\bar{p})$ solves the system
 \begin{equation}
\left\{
  \begin{array}{lcl}
    \dot{x}(s) & = &\nabla_p H(x(s),p(s)), \\
    -\dot{p}(s)& \in & \partial_xH(x(s),p(s),
  \end{array}
  \right.
  \end{equation}
 in $[0,\mathcal{T}(x)]$ with $\bar{x}(0) = x$ and $\bar{p}(0) = -\zeta$.
 
 Since $\zeta \in N^P_{\mathcal{R}(\mathcal{T}(x))}(x)$, by Theorem \ref{THM7} we have $-\bar{p}(t) \in N^P_{\mathcal{R}(\mathcal{T}(\bar{x}(t)))}(\bar{x}(t))$ for all $t\in [0,\mathcal{T}(x)]$.
 
 Set $r_1: = \mathcal{T}(x) -\mathcal{T}(y)$ and $x_1 := \bar{x}(r_1)$. We have
\begin{eqnarray}
\label{Eq55}
\langle \zeta,y-x\rangle &=& \langle \bar{p}(r_1) -\bar{p}(0), y - x\rangle + \langle -\bar{p}(r_1), y -x_1\rangle + \langle -\bar{p}(r_1),x_1 - x\rangle\nonumber\\
&=&: (III)+(IV) +(V).
\end{eqnarray}
We first consider (III). We have
\begin{eqnarray}
\label{Eq56}
(III) &=& \int_0^{r_1} \langle \dot{\bar{p}}(s),y-x\rangle ds \le K_4\int_0^{r_1}|\bar{p}(0)||y-x|ds = K_4|\zeta||y-x|r_1 \nonumber \\
&\le& K_4|\zeta| \left( |y-x|^2 + |\mathcal{T}(y)-\mathcal{T}(x)|^2 \right),
\end{eqnarray}
for some suitable constant $K_4 := K_4(\delta)$.

Since $-\bar{p}(r_1) \in N^P_{\mathcal{R}(\mathcal{T}(x_1))}(x_1)$ and $\mathcal{R}(\mathcal{T}(x_1))$ is $\varphi_0$ -convex, there has
\begin{eqnarray}
\label{Eq57}
(IV) &\le& \varphi_0|\bar{p}(r_1)||y-x_1|^2 \le \varphi_0|\bar{p}(r_1)|(|y-x|+|x-x_1|)^2\nonumber\\
&\le& K_5|\zeta| \left( |y-x|^2 + |\mathcal{T}(y) - \mathcal{T}(x)|^2\right),
\end{eqnarray}
for some suitable constant $K_5 := K_5(\delta,\varphi_0)$. 

We now consider (V).  We have
\begin{equation}
\label{Eq58}
(V) = \int_0^{r_1} \langle \bar{p}(s) - \bar{p}(r_1),\dot{\bar{x}}(s)\rangle ds + \int_0^{r_1}\langle -\bar{p}(s),\dot{\bar{x}}(s)\rangle ds.
\end{equation}
By the sublinear property of $F$ and the fact that  $\mathcal{R}(\delta)$ is compact, there is some constant $K_6 := K_6(\delta)$ such that $|\dot{\bar{x}}(s)| \le K_6$ for all $s\in [0,\mathcal{T}(x)]$. Using Lemma \ref{CK}, we have for all $s\in [0,r_1]$, 
$$\langle \bar{p}(s) - \bar{p}(r_1),\dot{\bar{x}}(s)\rangle \le |\bar{p}(s) - \bar{p}(r_1)||\dot{\bar{x}}(s)| \le K_7|r_1-s||\bar{p}(r_1)| \le K_8|\zeta|r_1,$$
for some suitable constants $K_7 := K_7(\delta), K_8 := K_8(\delta)$. Therefore, we have 
\begin{equation}
\label{Eq59}
\int_0^{r_1} \langle \bar{p}(s) - \bar{p}(r_1),\dot{\bar{x}}(s)\rangle ds \le K_8|\zeta|r^2_1 = K_8|\zeta||\mathcal{T}(y) - \mathcal{T}(x)|^2.
\end{equation}
To estimate the second term in the right-hand side of (\ref{Eq58}), we first note that, for all $s\in [0,\mathcal{T}(x)]$,
$$\langle - \bar{p}(s),\dot{\bar{x}}(s)\rangle = -H(\bar{x}(s),\bar{p}(s)) = h(\bar{x}(s),-\bar{p}(s)) = h(\bar{x}(0),-\bar{p}(0)) = h(x,\zeta).$$
Hence
\begin{equation}
\label{Eq510}
\int_0^{r_1} \langle -\bar{p}(s),\dot{\bar{x}}(s)\rangle ds = h(x,\zeta)(\mathcal{T}(x) - \mathcal{T}(y)).
\end{equation}
From (\ref{Eq55}) - (\ref{Eq510}), we obtain (\ref{Eq51}) for the case (ii).

We now progress as Step 2 in the proof of Theorem 3.7 in \cite{CMW} and conclude that there exists a continuous function $\varphi$ such that the epigraph of $\mathcal{T}_{|\mathcal{S}(\delta)}$ is $\varphi$-convex. The proof is complete.
 \end{proof}
 The following are some examples in which assumption (Q) holds true.
 \begin{Example} 
 (a)   Let $F(x) = \{Ax+u:u\in U\}$ for all $x\in \R^n$, where $A\in \mathbb{M}_{n\times n}(\R)$ and $U$ is a nonempty compact convex subset of $\R^n$. Let the target $\mathcal{K}$ be a closed, convex subset of $\R^n$  with $h(x,\zeta) \le 0$ for all $x\in \mathcal{K}$ and $\zeta \in N^P_{\mathcal{K}}(x)$. Then $\mathcal{R}(t)$ is convex for any $t>0$ (see Proposition 3.1 in \cite{CMW}). In this case, assumption (Q) holds for any $\delta >0$ and $\varphi_0 = 0$. 
 
 (b) Let $\mathcal{K}=\{0\}$ and   $F(x) = \{f(x) + g(x)u: u\in [-1,1]^m\}$ for all $x\in \R^2$, where $f:\R^2\to \R^2, g:\R^2 \to \mathbb{M}_{2\times m}(\R)$, $m=1$ or $m=2$, are of class $C^{1,1}$ (with Lipschitz constant $L$) and
 \begin{itemize}
 \item[(i)] $f(0) = 0$,
 \item[(ii)] $\mathrm{rank}[g_i(0),Df(0)g_i(0)] = 2$, for $i=1,m$  where $g=(g_1,g_m)$,
 \item[(iii)] $Dg(0) = 0$.
 \end{itemize}
 Then there exists $\tau >0$ depending only on $L, f(0), g(0)$ such that $\mathcal{R}(t)$ is (strictly) convex for all $0<t<\tau$ (see Theorem 5.1 in \cite{CK13}). Therefore, (Q) holds true for $\delta = \tau$ and $\varphi_0 = 0$.
 \end{Example}
 We are now going to provide conditions to ensure (Q) for differential inclusion which may not admit parameterizations with smooth functions using a result given by Pli\'s \cite{AP}. We first give some discussions on the assumptions were given in \cite{AP}.
\begin{Definition}
For a given real number $a>0$, a subset $S$ of $\R^n$ is called $a$-regular if for all points $x_0,x_1\in S$ and number $\lambda \in (0,1)$, the closed ball 
$$\{x\in \R^n: |x-\lambda x_1 - (1-\lambda)x_0| \le a\lambda(1-\lambda)|x_1-x_0|^2\}$$
is contained in $S$.
\end{Definition} 
Note that if $S$ is an $a$-regular set for some $a>0$, then so is $-S$. Moreover, any $a$-regular set is convex. A singleton is an $a$- regular set for any $a>0$.

Let $S_1,S_2 \subset \R^n$ be compact. The Hausdorff distance between $S_1$ and $S_2$ is defined as
$$\mathrm{dist}_{\mathcal{H}}(S_1,S_2) := \mathrm{max}\left\{\mathrm{dist}^{+}_{\mathcal{H}}(S_1,S_2),\mathrm{dist}^{+}_{\mathcal{H}}(S_2,S_1)   \right\},$$
where $\mathrm{dist}^{+}_{\mathcal{H}}(S,S') := \inf\{\varepsilon: S\subset S' +\varepsilon \mathbb{B}\}$.

The following class of multifunctions was introduced in \cite{AP}.
\begin{Definition}
A multifunction $F:\R^n \rightrightarrows \R^n$ is said to be of class $\mathcal{L}$ if there exists a positive constant $C$ such that for any points $x,y \in \R^n$ and for any number $\lambda \in [0,1]$, we have
\begin{equation}
\label{Eq511}
\mathrm{dist}_{\mathcal{H}}(F(\lambda x + (1-\lambda)y),\lambda F(x) + (1-\lambda)F(y)) \le C\lambda(1-\lambda)|x-y|^2.
\end{equation}
\end{Definition}
Observe that if $F:\R^n \rightrightarrows \R^n$ is of class $\mathcal{L}$ then so is $-F$. Moreover, the condition (\ref{Eq511}) is equivalent to the fact that the function $x\mapsto H(x,p)$ is both semiconcave and semiconvex for all $p\in \R^n$, that is, 
\begin{itemize}
\item[(H3)] $x \mapsto H(x,p)$ is of class $C^{1,1}$ for all $p\in \R^n$.
\end{itemize}
Note that this fact is also mentioned in \cite{PCPW}. \\
Let $A$ be a compact convex set and $p$ be a non-vanishing vector. Denote by $w(p,A)$ such a point of $A$ that
$$\langle w(p,A),p\rangle = \mathrm{max}_{w\in A} \langle w,p\rangle.$$
For given compact, convex sets $A,B$, we define
$$s(A,B) :=  \mathrm{max}_{w\ne 0} |w(p,A) - w(p,B)|.$$
In \cite{AP}, the following assumption was made on the multifunction $F$, for $x,y \in \R^n$
\begin{equation}
\label{SD}
s(F(x),F(y)) \le \kappa |x-y|,
\end{equation}
for some constant $\kappa >0$.\\
The assumption (H2) implies that the argmax set of $v\mapsto \langle v,p\rangle$ over $v\in F(x)$, $x,p\in \R^n, p\ne 0$, is singleton which equals $\nabla_pH(x,p)$. Thus, for $x,p \in \R^n, p \ne 0$,
$$w(p,F(x)) = \nabla_p H(x,p).$$
We have for $x,p \in \R^n$,
$$s(F(x),F(y)) = \mathrm{max}_{p\ne 0} |\nabla_pH(x,p) - \nabla_pH(y,p)|.$$
Then again by (H2), (\ref{SD}) holds locally. Therefore, (\ref{SD}) can be seen as a consequence of (H2).

For our result, we need the following technical lemma.
\begin{Lemma} \label{LemCv}
Let $A,B \subset \R^n$ be such that $A\subset B$ and $\mathrm{bdry}B \subset A$. If $A$ is convex, then so is $B$.
\end{Lemma}
 \begin{proof}
 Assume to the contrary that $B$ is not convex. Then there exist $x,y \in B$ such that $[x,y]\setminus B \ne \emptyset$. There also exist $x_1,y_1 \in [x,y]\cap \mathrm{bdry}B$ such that $[x_1,y_1]\setminus B \ne \emptyset$. Since $\mathrm{bdry}B \subset A$ and $A$ is convex, we have $[x_1,y_1] \subset A$. Hence $[x_1,y_1]\subset B$ due to $A\subset B$. This contradiction implies that $B$ is convex.
 \end{proof}
 \begin{Proposition} \label{ProCv}
 Assume (F), (H) and (H3). Suppose that $\mathcal{K}$ is compact and that, for some $a>0$, $\mathcal{K}$ and $F(x)$ are $a$-regular for all $x\in \R^n$. Then there exists $\tau>0$ such that $\mathcal{R}(t)$ is convex for all $t\in [0,\tau]$. 
 \end{Proposition}
 \begin{proof}
 For $T>0$, let $\mathcal{A}_1(T)$ be the attainable set from $\mathcal{K}$ at the time $T$ for the reversed differential inclusion
   \begin{equation}\label{RDI}
 \left\{
  \begin{array}{lcl}
    \dot{y}(t) & \in &-F(y(t)),\quad\quad \mathrm{a.e.}\,\, t>0\\
    y(0)& = & x \in \R^n
  \end{array}
  \right.
 \end{equation}
 that is,
 $$ \mathcal{A}_1(T): = \{y(T): y(\cdot)\,\text{solves}\,(\ref{RDI})\,\text{with}\,x\in \mathcal{K}\}.$$
 It is easy to see that
 \begin{itemize}
 \item[(i)] $\mathcal{A}_1(T) \subset \mathcal{R}(T)$,
 \item[(ii)] $\mathrm{bdry}\mathcal{R}(T) \subset \mathcal{A}_1(T)$.
 \end{itemize}
 As shown in \cite{AP} (see also Corollary 3.12 in \cite{PCHF}), there exists a number $\tau >0$ such that $\mathcal{A}_1(t)$ is convex for all $t\in [0,\tau]$. From Lemma \ref{LemCv}, $\mathcal{R}(t)$ is convex for all $t\in [0,\tau]$. This ends the proof.
 \end{proof}
 The main result of this section is stated as follows.
 \begin{Theorem}\label{PCV}
 Assume (F), (H) and (H3). Suppose that $\mathcal{K}$ is compact and for some $a>0$, $\mathcal{K}$ and $F(x)$ are $a$-regular for all $x\in \R^n$.  If $\mathcal{T}$ is continuous in $\mathcal{R}$, then there exist a number $\tau >0$  and a continuous function $\varphi$ such that the epigraph of $\mathcal{T}_{|\mathcal{S}(\tau)}$ is $\varphi$-convex.
 \end{Theorem}
 \begin{proof}
 It follows immediately from Proposition \ref{ProEpi} and \ref{ProCv}.
 \end{proof}
 \begin{Corollary}
 Under the same hypotheses of Theorem \ref{PCV}, the minimum time function $\mathcal{T}$ satisfies all the properties listed in Proposition \ref{ProPRF}.
 \end{Corollary}
 \subsection*{Acknowledgments} The author wish to express their sincere thanks to the anonymous referees for their helpful suggestions and comments which improved the original manuscript. He would also like to express his gratitude to Professor Giovanni Colombo for his valuable discussions and suggestions  from the beginning stage of this paper.
 
 The paper was supported by funds allocated to the implementation of the international co-funded
project in the years 2014-2018, 3038/7.PR/2014/2, and by the EU grant PCOFUND-GA-2012-
600415.

\end{document}